\documentclass[11pt,dvipdfmx,a4paper]{article}
\usepackage[margin=20mm]{geometry}
\usepackage{amssymb,amsthm,enumerate,color}
\usepackage{amsmath,fancybox,amsfonts,ascmac,bm,ulem}
\usepackage{graphicx}
\usepackage{tikz}
\numberwithin{equation}{section}

\usepackage[colorlinks=true,linkcolor=blue,citecolor=blue]{hyperref}
\theoremstyle{definition}
\newtheorem{thm}{Theorem}[section]
\newtheorem{defn}[thm]{Definition}
\newtheorem{ex}[thm]{Example}
\newtheorem{cor}[thm]{Corollary}
\newtheorem{lem}[thm]{Lemma}
\newtheorem{prop}[thm]{Proposition}

\newtheorem{rem}[thm]{Remark}

\newcommand{\R}{\mathbb{R}}   
\newcommand{\C}{\mathbb{C}}   
\newcommand{\N}{\mathbb{N}}   
\newcommand{\Z}{\mathbb{Z}}

\newcommand{\calH}{{\mathcal{H}}}

\newcommand{\calP}{{\mathcal{P}}}

\newcommand{\calZ}{{\mathcal{Z}}}

\newcommand{\rmd}{{\rm d}}

\title{Generalized Meixner-type free gamma distributions: \\
convolution formulas and potential correspondence}

\author{Noriyoshi Sakuma and Yuki Ueda}

\begin{document}
\maketitle  

\begin{abstract}
We introduce and study a class of generalized Meixner-type free gamma distributions $\mu_{t,\theta,\lambda}$ ($t,\theta>0$ and $\lambda\ge 1$), which includes both the free gamma distributions introduced by Anshelevich and certain scaled free beta prime distributions introduced by Yoshida. 
We investigate fundamental properties and mixture structures of these distributions. In particular, we consider the Gibbs distribution $\frac{1}{\mathcal{Z}_{t,\theta,\lambda}} \exp\{-V_{t,\theta,\lambda}(x)\}$ associated with a family of potentials $V_{t,\theta,\lambda}$, and show that $\mu_{t,\theta,\lambda}$ maximizes Voiculescu's free entropy with potential $V_{t,\theta,\lambda}$ for parameters $t,\theta>0$ and $1\le \lambda<1+t/\theta$. This result substantially extends the range of classcal-free correspondences obtained the potential function, differing from those arising from the Bercovici-Pata bijection. Moreover, we identify algebraic relations involving noncommutative random variables distributed as free gamma distributions.
\end{abstract}



\section{Introduction}
\label{sec:intro}

In free probability theory, the Bercovici-Pata bijection, which connects infinitely divisible laws with free counterparts, is often discussed in relation to its correspondence with classical probability \cite{BNT06}. A recurring issue in this context is the inadequate correspondence of the gamma distribution, which has hindered a comprehensive understanding of its free counterpart.  
The image of the class of classical gamma distributions under the Bercovici-Pata bijection was introduced by P\'{e}rez-Abreu and Sakuma \cite{PAS08}, but its distributional properties are not well understood (see e.g. \cite{HT14}).  

Meanwhile, an alternative approach was developed independently, focusing on orthogonal polynomials to investigate the properties of the ``free'' gamma distribution \cite{Ans03, BB06}. However, the ``free'' gamma distribution remains relatively underexplored, particularly in terms of its interpretation and characterization, leaving significant scope for further investigation. This gap may stem from inadequate parameter consideration and the lack of a proper examination of the distribution family as a whole. In this paper, we introduce a new class of probability distributions that extends the framework of the ``free'' gamma distribution. By refining the parameterization and exploring novel correspondences, we aim to overcome the limitations of prior research and provide deeper insights into the structure and potential applications of these distributions.

Before providing explanations, we introduce the notation used in this paper. We denote by $\calP(K)$ the set of all Borel probability measures on $\R$ whose support is contained in $K\subset \R$. In this paper, we frequently take $K$ to be the real line $\R$, the nonnegative real line $\R_{\ge0}:=[0,\infty)$, or the positive real line $\R_{>0}:=(0,\infty)$. Moreover, we adopt the following notational conventions.

\begin{itemize}
\setlength{\itemsep}{-\parsep}
\item ({\it Dilation}) For $\mu \in \mathcal{P}(\mathbb{R})$ and $c \neq 0$, we define $D_c(\mu)$ as the measure given by
\[
D_c(\mu)(B) := \mu(\{x/c : x \in B\}), \qquad \text{for all Borel sets } B \subset \mathbb{R}.
\]

\item ({\it Power of measure}) For $\mu \in \mathcal{P}(\mathbb{R}_{\ge 0})$ and $c > 0$, we define $\mu^{\langle c \rangle}$ by
\[
\mu^{\langle c \rangle}(B) := \mu(\{x^{1/c} : x \in B\}), \qquad \text{for all Borel sets } B \subset \mathbb{R}_{\ge 0}.
\]

\item ({\it Reversed measure}) For $\mu \in \mathcal{P}(\mathbb{R}_{\ge 0})$ with $\mu(\{0\}) = 0$, we define $\mu^{\langle -1 \rangle}$ by
\[
\mu^{\langle -1 \rangle}(B) := \mu(\{x^{-1} : x \in B\}), \qquad \text{for all Borel sets } B \subset \mathbb{R}_{>0}.
\]
\end{itemize}
Moreover, for $w \in \C\setminus\{0\}$, we define $\sqrt{w} := |w|^{\frac{1}{2}} e^{i\frac{\arg(w)}{2}}$ with $\arg(w) \in (0,2\pi)$. 

First, we briefly recall the definition of the classical gamma distribution, which is widely used in classical probability theory. The classical gamma distribution is a two-parameter family of probability distributions, defined as  
$$
\gamma(t, \theta) := \frac{1}{\theta^t \Gamma(t)} x^{t-1} e^{-\frac{x}{\theta}} \mathbf{1}_{(0,\infty)}(x) \, {\rm d}x, \qquad t, \theta > 0,
$$  
where \( t \) is the shape parameter and \( \theta \) is the scale parameter. The gamma distribution is known to be infinitely divisible, and its characteristic function can be expressed in terms of the L{\'e}vy-Khintchine representation as  
$$
\int_{\R} e^{izx} \gamma(t, \theta)({\rm d}x)
= \exp \left[ \int_{(0,\infty)} \left(e^{izx} - 1\right) \frac{t e^{-\frac{x}{\theta}}}{x} \, \rmd x \right], \qquad z \in \mathbb{R}
$$  
where the L{\'e}vy measure is given by  
$$
\nu(\rmd x) = \frac{t e^{-\frac{x}{\theta}}}{x} \mathbf{1}_{(0,\infty)}(x) \rmd x.
$$
It includes many distributions, such as the exponential distribution, Erlang distribution, and chi-squared distribution.
In addition, the following properties hold: 
\begin{itemize}
\setlength{\itemsep}{-\parsep}
\item $\gamma(1, \theta)^{\ast t} =\gamma(t, \theta)$ for all $t,\theta>0$;
\item $\gamma(t_{1}, \theta)\ast\gamma(t_{2}, \theta) = \gamma(t_{1}+t_{2},\theta)$ for all  $t_1,t_2,\theta>0$;
\item $D_{\theta}(\gamma(t,1))=\gamma(t,\theta)$ for all $t,\theta>0$.
\end{itemize} 

The first property motivates the consideration of the gamma L{\'e}vy process, where the shape parameter can also be interpreted as a ``time parameter'' in the context of stochastic processes. 
The second property is known as the reproductive property of probability distributions. The third property shows that $\theta$ is a scare parameter for gamma distributions.
Moreover, the gamma distribution has been extensively characterized in the literature; see, for example \cite{Mos85, B92, Sat13}.
Thus, the classical gamma distribution has attracted significant interest from many fields, such as probability, mathematical statistics, Bayesian statistics, econometrics, queueing theory and so on.

Given the richness of its structure, it was a natural step to investigate its analogue in free probability theory. In 2003, Anshelevich introduced the free gamma distribution as a subfamily of the free Meixner class (see \cite[Page 238]{Ans03}). More precisely, the {\it Meixner-type free gamma distribution}\footnote{In 2008, Pérez-Abreu and Sakuma \cite{PAS08} introduced another type of free gamma distribution via the Bercovici--Pata bijection, which maps the class of infinitely divisible distributions to their free counterparts. Therefore, we distinguish between two types of free gamma distributions by name.}, denoted by $\eta(t,\theta)$, is defined as the probability measure whose R-transform\footnote{Our R-transform differs from that given in \cite[Page 238]{Ans03} by a multiplicative factor of $z$.} is given by
\begin{align}\label{eq:R_eta(t,theta)}
R_{\eta(t,\theta)}(z) = \int_{\R} \left(\frac{1}{1-zx}-1\right) \frac{t\theta k_{\theta,1}(x)}{x}{\rm d}x = \frac{t}{2}(1-\sqrt{1-4\theta z}), \quad z\in \C^-,
\end{align}
where $t>0$ is the time (or shape) parameter, $\theta>0$ is the scale parameter
and the function $k_{\theta,\lambda}$ is defined as
$$
k_{\theta,\lambda}(x):= \frac{\sqrt{(a^+-x)(x-a^-)}}{2\pi \theta x} \mathbf{1}_{(a^-,a^+)}(x), \qquad \lambda \ge1,
$$
with $a^\pm :=\theta(\sqrt{\lambda}\pm 1)^2$. For all $t,\theta>0$, the measure $\eta(t,\theta)$ is freely infinitely divisible, which serves as the analogue of infinite divisibility in free probability. Moreover, the following properties hold:
\begin{itemize}
\setlength{\itemsep}{-\parsep}
\item $\eta(1,\theta)^{\boxplus t}=\eta(t,\theta) $ for all $t,\theta>0$;
\item $\eta(t_1,\theta) \boxplus \eta(t_2,\theta) =  \eta(t_1+t_2,\theta)$ for all $t_1,t_2,\theta>0$;
\item $D_\theta(\eta(t,1))=\eta(t,\theta)$ for all $t,\theta>0$,
\end{itemize}
where $\boxplus$ denotes the free additive convolution, and $\mu^{\boxplus t}$ is the free convolution power of $\mu\in \calP(\R)$. See Section \ref{subsec_freeprob} or \cite{BV1993} for the above concepts in free probability.

In this paper, we introduce a generalized family of probability measures that includes the Meixner-type free gamma distributions $\eta(t, \theta)$. Specifically, we consider the family $\{\mu_{t, \theta, \lambda} : t, \theta > 0,\ \lambda \ge 1\} \subset \calP(\R_{\ge0})$, where each measure is defined via its R-transform as
$$
R_{\mu_{t,\theta,\lambda}}(z) = \int_{\R}\left(\frac{1}{1-zx}-1\right) \frac{t k_{\theta,\lambda}(x)}{x}{\rm d}x, \quad z\in \C^-.
$$
We call the measure $\mu_{t, \theta, \lambda}$ the {\it generalized Meixner-type free gamma distribution}. It is straightforward to verify that 
$$
\mu_{t\theta, \theta, 1} = \eta(t, \theta), \qquad t,\theta>0,
$$
and therefore the family indeed extends the class of Meixner-type free gamma distributions. Moreover, the parameter $t$ admits a natural interpretation as a free convolution power:
$$
\mu_{t,\theta,\lambda} = \mu_{1,\theta,\lambda}^{\boxplus t}, \qquad t>0.
$$
Note that the measure $\mu_{t,\theta,\lambda}$ coincides with the centered free Meixner distribution, up to a shift (see Section~\ref{sec2.2} and Proposition \ref{prop:relation_freeMeixner}). Below, we outline the structure of the paper and summarize our main results.

In Section \ref{sec3}, we study various distributional properties of the generalized Meixner-type free gamma distributions, including their density, existence of atoms, and moments (see Section~\ref{sec3.1}), as well as free selfdecomposability and unimodality (see Section~\ref{sec3.3}). Furthermore, we study the free L\'{e}vy processes related to the measures $\mu_{t,\theta,\lambda}$ in Sections~\ref{sec:3.4}--\ref{sec:3.5}.

In Section~\ref{sec4}, we derive formulas involving the free multiplicative convolution $\boxtimes$.

\begin{thm}[{\bf Free convolution formula for the measure $\mu_{t,\theta,\lambda}$}, see Theorem \ref{thm:formula_free_multiplicative}]\label{mainthm2}
Let $t,\theta>0$ and $\lambda \ge1$. Then the following properties hold.
\begin{enumerate}[\rm (1)]
\setlength{\itemsep}{-\parsep}
\item 
For $\lambda > 1$, the measure $\mu_{t,\theta,\lambda}$ can be expressed in two equivalent forms:
\begin{align*}
\mu_{t,\theta,\lambda}  
 = D_{t(\lambda-1)}\left(\pi_{1,\frac{t}{\theta(\lambda-1)}} \boxtimes (\pi_{1,1+\frac{t}{\theta}})^{\langle -1 \rangle}\right) = \mu_{t,\theta,1} \boxtimes \pi_{q^{-1},q}.
\end{align*}
Here, $q = \frac{t}{\theta(\lambda-1)}$ and $\pi_{\theta,\lambda}$ is the probability measure defined by
\[
\pi_{\theta,\lambda}({\rm d}x) = \max\{0, 1-\lambda\}\delta_0 + k_{\theta,\lambda}(x)\,{\rm d}x.
\]
\item In particular, $\mu_{t,\theta,1+t/\theta}= \mu_{t,\theta,1} \boxtimes \pi_{1,1}$. Hence the measure $\mu_{t,\theta,1+t/\theta}$ belongs to the class of free compound Poisson distributions.
\end{enumerate}
\end{thm}
According to Theorem~\ref{mainthm2} (1), for $\lambda > 1$, the measure $\mu_{t, \theta, \lambda}$ coincides with a suitably scaled free beta prime distribution $f\beta'(a, b)$ introduced in \cite{Yos20}. As a consequence, various properties of the free beta prime distribution can be derived from the results established for $\mu_{t, \theta, \lambda}$. See Section~\ref{sec6.2} for further details.

In Section \ref{sec5}, for $t,\theta>0$ and $\lambda\ge 1$, we consider the potential function $V_{t,\theta,\lambda}$ defined by
\begin{align*}
V_{t,\theta,\lambda} (x):= 
\begin{cases}
\left(2+\dfrac{t}{\theta}\right) \log x + \dfrac{t^2}{\theta x}, & \lambda=1,\\
\left(1-\dfrac{t}{\theta(\lambda-1)} \right) \log x + \left( 1 + \dfrac{t\lambda}{\theta(\lambda-1)}\right) \log (x+t(\lambda-1)), & \lambda>1.
\end{cases}
\end{align*}
We first derive the explicit form of the Gibbs measure
$$\rho_{t,\theta,\lambda}({\rm d}x) = \frac{1}{\mathcal{Z}_{t,\theta,\lambda}}\exp\{-V_{t,\theta,\lambda}(x)\}{\rm d}x,$$ 
associated with the potential $V_{t,\theta,\lambda}$, where the normalization constant is given by $\mathcal{Z}_{t,\theta,\lambda}= \int \exp\{- V_{t,\theta,\lambda}(x) \}{\rm d}x$.
\begin{itemize}
\setlength{\itemsep}{-\parsep}
\item For $\lambda=1$, we obtain
$$
\rho_{t,\theta,1}({\rm d}x) = \frac{(\frac{t^2}{\theta})^{1+\frac{t}{\theta}}}{\Gamma(1+\frac{t}{\theta})} x^{-(2+ \frac{t}{\theta})} e^{-\frac{t^2}{\theta x}} {\rm d}x.
$$
\item For $\lambda>1$, we get
$$
\rho_{t,\theta,\lambda}({\rm d}x)= \frac{(t(\lambda-1))^{1+\frac{t}{\theta}}}{B(\frac{t}{\theta(\lambda-1)},1+\frac{t}{\theta})} x^{-1 + \frac{t}{\theta(\lambda-1)}} \left(x+t(\lambda-1)\right)^{-1-\frac{t\lambda}{\theta(\lambda-1)}}{\rm d}x,
$$
where $B(a,b)$ denotes the beta function with parameter $a,b>0$.
\end{itemize}
The measure $\rho_{t,\theta,\lambda}$ coincides with the beta prime distribution when $\lambda > 1$. This close resemblance between classical and free analogues is striking and highlights deep structural parallels.

\begin{thm} ({\bf Convolution formula for the Gibbs measure $\rho_{t,\theta,\lambda}$}, see Theorem \ref{thm:rho_multiformula}) Let us consider $t,\theta>0$ and $\lambda \ge 1$. Then the following properties hold.
\begin{enumerate}[\rm (1)]
\setlength{\itemsep}{-\parsep}
\item For $\lambda>1$, we have
\begin{align*}
\rho_{t,\theta, \lambda} &= D_{t(\lambda-1)} \left( \gamma\left(\frac{t}{\theta(\lambda-1)},1\right) \circledast \gamma\left(1+\frac{t}{\theta},1\right)^{\langle -1\rangle}\right)\\
&=\rho_{t,\theta,1} \circledast \gamma\left(\frac{t}{\theta(\lambda-1)}, \frac{\theta(\lambda-1)}{t}\right),
\end{align*}
where $\circledast$ is the classical multiplicative convolution.
\item In particular, $\rho_{t,\theta, 1+t/\theta}= \rho_{t,\theta,1} \circledast \gamma(1,1)$. Hence the measure $\rho_{t,\theta, 1+t/\theta}$ belongs to the class of mixture of exponential distributions.
\end{enumerate}
\end{thm}

Next, we show that for all $t, \theta > 0$ and $1 \le \lambda < 1 + t/\theta$, the measure $\mu_{t,\theta,\lambda}$ uniquely maximizes the free entropy associated with the potential $V_{t,\theta,\lambda}$. In other words, $\mu_{t,\theta,\lambda}$ serves as the equilibrium measure in this variational framework.

\begin{thm}[{\bf Free entropy associated with $V_{t,\theta,\lambda}$}, see Theorem \ref{thm:free_entropy}] 
For $t,\theta>0$ and $1\le \lambda <1+ t/\theta$, we have
$$
\mu_{t,\theta,\lambda} = \text{argmax}\{ \Sigma_{V_{t,\theta,\lambda}}(\mu): \mu \in \calP(\R_{>0})\},
$$
where $\Sigma_{V_{t,\theta,\lambda}}(\mu)$ is the (Voiculescu's) free entropy:
$$
\Sigma_{V_{t,\theta,\lambda}}(\mu)=\iint_{\R_{>0}\times \R_{>0}} \log|x-y|\mu(\rmd x)\mu(\rmd y) - \int_{\R_{>0}} V_{t,\theta,\lambda}(x) \mu(\rmd x).
$$
In particular, for $a,b>1$, the measure $f\beta'(a,b)$ is the unique maximizer of the free entropy $\Sigma_{V_{a,b}}$, where 
$$
V_{a,b}(x)=(1-a)\log x+ (a+b)\log (1+x), \qquad x>0,
$$
see also Corollary \ref{cor:freebetaprime_entropy}.
\end{thm}

In Section \ref{sec:Fbgalg}, we investigate algebraic properties of noncommutative random variables $G^{(p)} \sim \eta(p,1)$, which we refer to as {\it Meixner-type free beta-gamma algebras}. Consider freely independent\footnote{See e.g. \cite{NS06} and references therein for background on free independence.} noncommutative random variables  $G_1^{(p)} \sim \eta(p,1)$ and $G_2^{(q)} \sim \eta(q,1)$. Since the Meixner-type free gamma distributions satisfy the convolution identity $\eta(p,1)\boxplus \eta(q,1)= \eta(p+q,1)$, we obtain the following distributional identity:
$$
G_1^{(p)} + G_2^{(q)} \overset{\rmd}{=} G_3^{(p+q)} \sim \eta(p+q,1),
$$
where $X \overset{\rmd}{=} Y $ denotes equality in distribution. This identity highlights the additive stability of the Meixner-type free gamma distributions under free convolution. It is worth noting that the three parameters $(t, \theta, \lambda)$ of the generalized Meixner-type free gamma distributions play a crucial role in understanding the algebraic structure of the random variables $G^{(p)}$.

\begin{thm} ({\bf Meixner-type free beta-gamma algebras}, see Section \ref{sec:Fbgalg}) We denote by $G^{(p)}\sim \eta(p,1)$ a noncommutative random variable and let us set free copies $\{G_1^{(p)},G_2^{(p)},\dots\}$ from $G^{(p)}$ for each $p>0$. Then
\begin{itemize}
\setlength{\itemsep}{-\parsep}
\item  $(G_2^{(q)})^{-\frac{1}{2}} G_1^{(p)} (G_2^{(q)})^{-\frac{1}{2}} \sim \eta(p,1)\boxtimes \eta(q,1)^{\langle-1 \rangle} = D_{\frac{1+q}{q^2}}\left( \mu_{p,1,1+\frac{p}{1+q}}\right)$ for $p,q>0$.
\item For any $p>0$ and $n\in \N$,
$$
\left(\frac{1}{G_1^{(p)}}+\frac{1}{G_2^{(p)}}+\cdots+ \frac{1}{G_{2^n}^{(p)}}\right)^{-1} \overset{\rmd}{=}\left(2^n+\frac{2^n-1}{p}\right)^{-2} G^{(2^np+2^n-1)}.
$$
\item In the case of $p=\frac{1}{2(m-1)}$ for some natural number $m\ge2$, we have
$$
\left(\frac{1}{G_1^{(p)}}+\frac{1}{G_2^{(p)}}\right)^{-1} \overset{\rmd}{=} \frac{1}{4m^2} (G_1^{(2p)} + \cdots + G_m^{(2p)}).
$$
\item The law $\mu_p$ of the Meixner-type free beta random variable 
$$
B^{(p)}:=\{(G_1^{(p)})^{-1} +(G_2^{(p)})^{-1}\}^{-\frac{1}{2}} (G_1^{(p)})^{-1} \{(G_1^{(p)})^{-1} +(G_2^{(p)})^{-1}\}^{-\frac{1}{2}}
$$ 
has the following R-transform
$$
R_{\mu_p}(z)= \frac{p(z-p^3)- \sqrt{(3p+2)^2 z^2 -2p^5 z + p^8}}{2z}, \qquad z \in \left(-\frac{p^3}{2(p+1)},0\right).
$$
Moreover, $\mu_p$ is not freely infinitely divisible for any $p>0$.
\end{itemize}
\end{thm}

In Section \ref{sec7}, using the method of {\it finite free probability} which is an approximation theory in free probability that has attracted attention in recent years (see \cite{Mar21,MSS22}), we demonstrate that the asymptotic behavior of the roots of some Jacobi and Bessel polynomials, as their degree becomes sufficiently large, can be understood through the generalized Meixner-type free gamma distributions. The key point that aids understanding is the use of the ``finite S-transform'' recently introduced by the second author \cite{AFPU24}. Specifically, it involves investigating the relationship between the finite S-transform of Jacobi polynomials or Bessel polynomials and the S-transform of the generalized Meixner-type free gamma distribution derived in Section \ref{sec4}.


\section{Preliminaries}

\subsection{Harmonic analysis in free probability}\label{subsec_freeprob}

In this paper, we employ the framework of free harmonic analysis as introduced by \cite{BV1993}; see also Chapter 3 in \cite{MS17}. To compare with classical probability theory, infinite divisibility is often used since there is the Bercovici-Pata bijection, a mapping that relates classical infinitely divisible distributions to their free counterparts, see \cite{BP99, BNT02, BNT06} for details.
From this perspective, we introduce the necessary tools for our analysis in the following sections.

A probability measure $\mu$ on $\R$ is called {\it freely infinitely divisible} if for any $n\in\N$ there exists a probability measure $\mu_{n}\in\calP(\R)$ such that
\begin{align*}
\mu = \underbrace{\mu_{n}\boxplus \dots \boxplus \mu_{n}}_{\text{$n$ times}},
\end{align*}
where $\boxplus$ denotes the {\it free additive convolution} which can be defined
as the distribution of sum of freely independent selfadjoint operators.
In this case, $\mu_n\in\calP(\R)$ is uniquely determined for each $n\in\N$.
The freely infinite divisible distributions
can be characterized as those admitting a L{\'e}vy-Khintchine representation in terms of R-transform which is the free analog of the cumulant transform $C_{\mu}(z) := \log (\widehat{\mu}(z))$, where $\widehat{\mu}$ is the characteristic function of $\mu\in\calP(\R)$.
This was originally established by
Bercovici and Voiculescu in \cite{BV1993} for all Borel probability measures.
To explain it, we gather analytic tools for free additive convolution $\boxplus$. In order to define
the \textit{R-transform} (or \textit{free cumulant transform}) $R_\mu$ of $\mu\in\calP(\R)$, we need to define its
{\it Cauchy-Stieltjes transform} $G_\mu$:
\begin{equation*}
G_\mu(z)=\int_{\R}\frac{1}{z-t}\,\mu(\rmd t), \qquad(z\in\C^+).
\end{equation*}
Note in particular that $\Im(G_\mu(z))<0$ for any $z$ in $\C^+$, and
hence we may consider the reciprocal Cauchy transform
$F_\mu\colon\C^{+}\to\C^{+}$ given by $F_{\mu}(z)=1/G_{\mu}(z)$.
For any $\mu\in\calP(\R)$ and any $\lambda>0$ there exist
positive numbers $\alpha,\beta$ and $M$ such that $F_{\mu}$ is univalent
on the set $\Gamma_{\alpha,\beta}:=\{z \in \C^{+} \,|\, \Im(z) >\beta,
|\Re(z)|<\alpha \Im(z)\}$ and such that
$F_{\mu}(\Gamma_{\alpha,\beta})\supset\Gamma_{\lambda,M}$.
Therefore the right inverse $F^{\langle -1 \rangle}_{\mu}$ of $F_{\mu}$ exists on
$\Gamma_{\lambda,M}$, and the R-transform (or free cumulant transform)
$ R_\mu$ is defined by
\begin{align*}
 R_{\mu}(w) =wF^{\langle -1 \rangle}_{\mu}\left(\frac{1}{w} \right)-1, \quad\text{for all $w$ such that
  $1/w \in \Gamma_{\lambda,M}$}.
\end{align*}
The free version of the L\'evy-Khintchine representation now amounts
to the statement that $\mu\in \calP(\R)$
is freely infinitely divisible if and only if there exist $a\ge 0$,
$\gamma\in\R$ and a L{\'e}vy measure\footnote{A (Borel-) measure $\nu$
  on $\R$ is called a L\'evy measure, if $\nu(\{0\})=0$ and
  $\int_{\R}\min\{1,x^2\}\,\nu(\rmd x)<\infty$.}
$\nu$ such that
\begin{align}
 R_{\mu}(w) = a w^{2}+\gamma w +
\int_{\R}\left(\frac{1}{1- w x}-1-w x \mathbf{1}_{[-1,1]}(x)\right)\nu(\rmd x),\qquad w\in\C^-.
\label{eqno1}
\end{align}
The triplet $(a,\gamma,\nu)$ is uniquely determined and referred to as
the {\it free characteristic triplet} for $\mu$, and the measure $\nu$ is referred to as the {\it  free L\'evy measure} for $\mu$. Recently free infinite divisibility has been proved for normal distributions \cite{BBLS11}, some of the boolean-stable distributions \cite{AH14}, some of the beta distributions and some of the gamma distributions, including the chi-square distribution and powers of random variables distributed as these distributions \cite{Has14, Has16} and generalized power distributions with free Poisson term \cite{MU20}.

As one of the most important subclass of freely infinitely divisible distributions, we introduce the concepts of {\it freely selfdecomposable distributions}. A probability measure $\mu$ on $\R$ is said to be freely selfdecomposable if for any $c\in (0,1)$ there exists $\mu_c \in \mathcal{P}(\R)$ such that $\mu= \mu_c \boxplus D_c(\mu)$. It is easy to see that every freely selfdecomposable distribution is freely infinitely divisible. Moreover, it is known that $\mu \in \mathcal{P}(\R)$ is freely selfdecomposable if and only if its free L\'{e}vy measure $\nu$ is formed by
$$
\nu({\rm d}x) = \frac{k(x)}{|x|} \mathbf{1}_{\R\setminus \{0\}}(x) {\rm d}x,
$$
where a function $k$ is nondecreasing on $(-\infty,0)$ and nonincreasing on $(0,\infty)$, see \cite{BNT02} for details. Examples and properties of freely selfdecomposable distributions were investigated by \cite{HST19, HT16, HU23, MS23}.

Let us consider $\mu \in \mathcal{P}(\R_{\ge0}) \setminus\{\delta_0\}$. We define the {\it S-transform} of $\mu$ by
$$
S_\mu(z) := \frac{1+z}{z} \Psi_\mu^{\langle-1\rangle}(z), \qquad z \in \Psi_\mu (i \C^+),
$$
where $\Psi_\mu$ is the moment generating function that is
$$
\Psi_\mu(z):= \int_{\R_{>0}} \frac{xz}{1-xz} \mu({\rm d}x), \qquad z\in \C\setminus \R_{\ge0},
$$
and $\Psi_\mu(i\C^+)$ is a region contained in the circle with diameter $(\mu(\{0\})-1,0)$. One can see that
\begin{align}\label{eq:Strans_dilation}
S_{D_c(\mu)}(z) = \frac{1}{c} S_\mu(z), \qquad c>0.
\end{align}
For $\mu,\nu\in \mathcal{P}(\R_{\ge0})\setminus\{\delta_0\}$, we obtain $S_{\mu\boxtimes \nu}=S_\mu S_\nu$ on the common domain in which three S-transforms are defined, where $\mu\boxtimes \nu$ is called the {\it free multiplicative convolution} which is the distribution of multiplication $\sqrt{X}Y\sqrt{X}$ of freely independent positive random variables $X\sim\mu$ and $Y\sim\nu$.
It is known that 
\begin{align}\label{eq:formula_R_S}
R_\mu(zS_\mu(z)) =z,
\end{align} 
for small enough $z$ in a neighborhood of $(\mu(\{0\})-1,0)$, see \cite{BV92, BV1993} for details.

\begin{ex}[Marchenko-Pastur distribution]\label{ex_MPlaw}
We define the function $k_{\theta,\lambda}$ as
$$
k_{\theta,\lambda}(x):= \frac{\sqrt{(a^+-x)(x-a^-)}}{2\pi \theta x} \mathbf{1}_{(a^-,a^+)}(x), 
$$
where $\theta,\lambda>0$ and $a^\pm  :=\theta(\sqrt{\lambda}\pm 1)^2$, respectively. The \textit{Marchenko-Pastur law} $\pi_{\theta,\lambda}$ with the scale parameter $\theta$ and the shape parameter $\lambda$ is defined as the probability measure given by
$$
\pi_{\theta,\lambda}({\rm d}x) := \max\{0, 1-\lambda\}\delta_0 + k_{\theta,\lambda}(x){\rm d}x.
$$
It is known that
$$
R_{\pi_{\theta,\lambda}}(z) = \frac{\theta\lambda z}{1-\theta z}, \qquad z\in \C^-,
$$
and
$$
S_{\pi_{\theta,\lambda}}(z) = \frac{1}{\theta(\lambda +z)}, \qquad z \text{ in a neighborhood of }  (-1+\pi_{\theta,\lambda}(\{0\}),0).
$$
From the form of R-transform, we notice that, for any $\theta,\lambda>0$, 
$$
\pi_{\theta,\lambda} =D_{\theta} (\pi_{1,1}^{\boxplus \lambda}).
$$
\end{ex}

According to \cite[Proposition 3.13]{HS07}, we have
\begin{align}\label{eq:S-trans_inverse}
S_{\mu^{\langle -1 \rangle}}(z) = \frac{1}{S_\mu(-z-1)}, \qquad z\in (-1,0).
\end{align}

\begin{ex}[Free positive stable law with index $1/2$]\label{ex:inverse_MP}
By \eqref{eq:S-trans_inverse}, for $\lambda\ge 1$,
\begin{align*}
S_{\pi_{\theta,\lambda}^{\langle -1\rangle}}(z)= \frac{1}{S_{\pi_{\theta,\lambda}}(-z-1)} = \theta(\lambda-1-z), \qquad z\in (-1,0).
\end{align*}
In particular, the measure $\pi_{1,1}^{\langle -1\rangle}$ has a probability density $\frac{\sqrt{4x-1}}{2\pi x^2} \mathbf{1}_{(1/4,\infty)}(x) \rmd x$ and is well known as the free positive stable law  with index $1/2$, introduced by \cite{BP99}.
\end{ex}

\subsection{Centered free Meixner distributions}
\label{sec2.2}

In this section, we introduce the three-parameter family $\{\nu_{s,a,b}: s\ge0, a \in \R, b\ge -1\} \subset \calP(\R)$ in which their Cauchy transform is given by
\begin{align}
G_{\nu_{s,a,b}}(z) &= \cfrac{1}{z-\cfrac{s}{z-a-\cfrac{s+b}{z-a-\cfrac{s+b}{\ddots}}}}\\
&= \frac{(s+2b)z+sa-s\sqrt{(z-a)^2-4(s+b)}}{2(bz^2+saz+s^2)}.  \label{eq:Cauchy_Meixner}
\end{align}
The measure $\nu_{s,a,b}$ is called the {\it centered free Meixner distribution}. According to \cite{SY01, Ans03, BB06}, $\nu_{s,a,b}$ is freely infinitely divisible whenever $b\ge 0$, and an integral representation for the R-transform of $\nu_{s, a,b}$ is given by
$$
R_{\nu_{s, a,b}}(z) =  \int_{\R} \frac{z^2}{1-zx}s \ w_{a,b}(x){\rm d}x,
$$
where 
$$
w_{a,b}(x) = \frac{1}{2\pi b}\sqrt{4b-(x-a)^2}\mathbf{1}_{[a-2\sqrt{b}, a+2\sqrt{b}]}(x)
$$ 
is the density of Wigner's semicircle law with mean $a\in \R$ and variance $b\ge0$. Note that, the above R-transform differs from \cite{SY01, Ans03, BB06} by a factor of $z$. In this case, one can see that $\nu_{s,a,b}= \nu_{1,a,b}^{\boxplus s}$.
By \cite[Equation (4)]{BB06}, the R-transform of $\nu_{s,a,b}$ admits the explicit form
\begin{align}\label{eq:R-trans_freeMeixner}
R_{\nu_{s,a,b}}(z)=\frac{2sz^2}{1-az+\sqrt{(1-az)^2-4bz^2}}, \qquad b\neq 0
\end{align}
and in the case $b=0$, it reduces to
$$
R_{\nu_{s,a,0}}(z) = \frac{sz^2}{1-az}.
$$
According to \cite[Theorem 3.2]{BB06}, the centered free Meixner law $\nu_{1,a,b}$ coincides with one of the following measures:
\begin{itemize}
\setlength{\itemsep}{-\parsep}
\item the Wigner's semicircle law if $a=b=0$;
\item the Marchenko-Pastur distribution if $b=0$ and $a\neq 0$;
\item the free Pascal (negative binomial) distribution if $b>0$ and $a^2>4b$;
\item the free gamma distribution if $b>0$ and $a^2=4b$;
\item the pure free Meixner distribution if $b>0$ and $a^2<4b$;
\item the free binomial distribution if $-\min\{\alpha, 1-\alpha\} \le b <0$, where $\alpha=\int_{\R} x^2\ \nu_{1,a,b}(\rmd x)$.
\end{itemize}

By \eqref{eq:R_eta(t,theta)} and \eqref{eq:R-trans_freeMeixner}, the Meixner-type free gamma distribution $\eta(t,\theta)$ can be expressed in terms of the centered free Meixner law (the free gamma distribution in the sense described above) as
$$
\eta(t,\theta) = \nu_{t\theta^2, 2\theta, \theta^2} \boxplus \delta_{t\theta}, \qquad t,\theta>0.
$$
More generally, we show that the generalized Meixner-type free gamma distribution can be represented as a centered free Meixner law under a shift (see Proposition \ref{prop:relation_freeMeixner}).

\subsection{Entropy functionals with potentials}\label{subsec_potential}

Assume that $V$ is a $C^1$-potential function $V$  satisfying 
$$
V(x) \ge (1+\delta)\log (x^2+1), \qquad x\in \R, \qquad \text{for some } \delta>0,
$$
and $\mathcal{Z} := \int e^{-V(x)}{\rm d}x<\infty$. 

By the Lagrangian multiplier method, it is known that the Gibbs distribution $\frac{1}{\mathcal{Z}} \exp\{-V(x)\}$
is a unique probability density which maximizes the Shannon entropy associated with the potential function $V$:
$$
H_V(p):= -\int p(x)\log p(x){\rm d}x - \int V(x) p(x){\rm d}x,
$$
among all probability density functions $p$ on $\R$.

According to \cite{Joh98}, it is known that for the above potential function $V$, the {\it free entropy functional} (see \cite{Voi93}):
$$
\Sigma_V(\mu) =  \iint \log |x-y| \mu ({\rm d}x) \mu ({\rm d}y) - \int V(x) \mu({\rm d}x),
$$
among all probability measures $\mu$ on $\R$, is known to be finite and have a unique maximizer $\mu_V$ (namely, the {\it equilibrium measure} of $\Sigma_V$). The support of $\mu_V$ is compact. Moreover, $\mu_V$ satisfies the following equation:
$$
\mathcal{H} \mu_V(x) = \frac{1}{2} V'(x), \qquad x\in \text{supp}(\mu_V),
$$
where $\mathcal{H}\mu$ is the {\it Hilbert transform} of a probability measure $\mu$ on $\R$, that is, 
$$
\mathcal{H}\mu (x) :=\text{p.v.} \int \frac{1}{x-y} \mu({\rm d} y)  = \lim_{\varepsilon \to 0} \left(\int_{-\infty}^{x-\varepsilon} + \int_{x+\varepsilon}^\infty\right)\frac{1}{x-y} \mu({\rm d} y), \quad x\in \R,
$$
see also \cite[Page 27, Theorem 1.3]{ST97} and \cite[(3.4)]{Bia03}.

In connection with the above discussion, Hasebe and Szpojankowski \cite{HS19} pointed out a correspondence between the measure that maximizes the Shannon entropy and the equilibrium measure of the free entropy, from the perspective of maximizing entropy functionals with a potential. We call it the {\it potential correspondence}\footnote{To the best of the author's knowledge, the exact name for this correspondence is unknown, and little is known about the mathematical facts.} in this paper. In \cite{HS19}, it was observed that the potential correspondence maps the classical generalized inverse Gaussian (GIG) distributions to the free GIG distributions introduced in \cite{Fer06}. Moreover, this correspondence maps the normal distributions $N(\mu,\sigma^2)$ to Wigner's semicircle laws $w_{\mu,\sigma^2}(x)\rmd x$ for $\mu\in \R$ and $\sigma>0$, and the gamma distribution $\gamma(\lambda,\theta)$ to the Marchenko-Pastur distributions $\pi_{\theta,\lambda}$ for $\theta>0$ and $\lambda\ge 1$.


\section{Generalized Meixner-type free gamma distributions}
\label{sec3}

Recall the definition of generalized Meixner-type free gamma distributions.
\begin{defn}\label{def:GFMG} Consider $t,\theta>0$ and $\lambda \ge 1$. The {\it generalized Meixner-type free gamma distribution} $\mu_{t,\theta,\lambda}$ is the probability measure whose R-transform is given by
$$
R_{\mu_{t,\theta,\lambda}}(z) = \int_{\R}\left(\frac{1}{1-zx}-1\right) \frac{t k_{\theta,\lambda}(x)}{x}{\rm d}x, \quad z\in \C^-.
$$
Here, $k_{\theta,\lambda}(x)$ denotes the density of the Marchenko-Pastur distribution, explicitly given by
$$
k_{\theta,\lambda}(x):= \frac{\sqrt{(a^+-x)(x-a^-)}}{2\pi \theta x} \mathbf{1}_{(a^-,a^+)}(x),
$$
with $a^\pm  :=\theta(\sqrt{\lambda}\pm 1)^2$. 
\end{defn}

\subsection{Relation with centered free Meixner distributions}

We establish an important connection between the centered free Meixner distributions $\nu_{s,a,b}$ (defined in Section~\ref{sec2.2}) and the generalized Meixner-type free gamma distributions $\mu_{t,\theta,\lambda}$.

\begin{prop}\label{prop:relation_freeMeixner}
For $t,\theta>0$ and $\lambda \ge1$, we have
\begin{align*}
\mu_{t,\theta,\lambda} = \nu_{t\theta\lambda, \theta(\lambda+1), \theta^2\lambda} \boxplus \delta_t. 
\end{align*}
\end{prop}
\begin{proof}
A direct computation shows that
\begin{align*}
R_{\nu_{t\theta\lambda, \theta(\lambda+1), \theta^2\lambda} \boxplus \delta_t}(z) 
&= R_{{\nu_{t\theta\lambda, \theta(\lambda+1), \theta^2\lambda}} } (z) + tz\\
&=\int_{\R} \frac{z^2}{1-zx} \cdot t\theta\lambda \cdot w_{\theta(\lambda+1),\theta^2\lambda}(x){\rm d}x + tz\\
&=\int_{\R} \frac{xz^2}{1-zx} \cdot t k_{\theta,\lambda}(x){\rm d}x+tz\\
&=tz \int_{\R} \frac{1}{1-zx} k_{\theta,\lambda}(x){\rm d}x\\
&=\int_{\R} \left(\frac{1}{1-zx}-1\right) \frac{tk_{\theta,\lambda}(x)}{x} {\rm d}x = R_{\mu_{t,\theta,\lambda}}(z),
\end{align*}
as desired.
\end{proof}

\begin{rem}
According to Proposition \ref{prop:relation_freeMeixner}, the generalized Meixner-type free gamma distribution $\mu_{t,\theta,\lambda}$ coincides, up to a shift, with the centered free Meixner distribution $\nu_{s,a,b}$ when the parameters satisfy
$$
s=t\theta\lambda, \quad a=\theta(\lambda+1), \quad \text{and} \quad b=\theta^2\lambda.
$$
In this setting, since $a^2\ge 4b$, the measure $\nu_{s,a,b}$ is either the free Pascal distribution (when $a^2>4b$) or the free gamma distribution (when $a^2=4b$), see \cite[Page 65]{BB06}. Consequently, the measure $\mu_{t,\theta,\lambda}$ can be interpreted as a suitably shifted version of either the free Pascal or the free gamma distribution.
\end{rem}

It follows from \eqref{eq:R-trans_freeMeixner} and Proposition \ref{prop:relation_freeMeixner} that
\begin{align}\label{eq:R-trans_theta}
R_{\mu_{t,\theta,\lambda}}(z) =t \cdot \frac{1+\theta(1-\lambda)z -\sqrt{(1+\theta(1-\lambda)z)^2-4\theta z}}{2\theta}.
\end{align}
Due to the above representation of R-transform, we can understand the second parameter $\theta$ for the measure $\mu_{t,\theta,\lambda}$. Since
\begin{align*}
R_{\mu_{t,\theta,\lambda}}(z) 
= \frac{1}{\theta} R_{\mu_{t,1,\lambda}}(\theta z)
=\frac{1}{\theta} R_{D_\theta(\mu_{t,1,\lambda})}(z) = R_{D_\theta(\mu_{t,1,\lambda})^{\boxplus \frac{1}{\theta}}}(z),
\end{align*}
we get 
$$
\mu_{t,\theta,\lambda}= D_\theta(\mu_{t,1,\lambda})^{\boxplus \frac{1}{\theta}}.
$$
We will explain the meaning of the third parameter $\lambda$ in Theorem \ref{thm:formula_free_multiplicative}.

\subsection{Density, atom and moments}
\label{sec3.1}

In this section, we investigate the density, atom and moments of $\mu_{t,\theta,\lambda}$. Thanks to \eqref{eq:Cauchy_Meixner} and  Proposition \ref{prop:relation_freeMeixner}, it is straightforward to see that
$$
G_{\mu_{t,\theta,\lambda}}(z)=\frac{(t+2\theta)z-t(t-\theta(\lambda-1)) -t \sqrt{(z-\alpha^-)(z-\alpha^+)}}{2\theta z (z+t(\lambda-1))}, \qquad z\in \C^+,
$$
where
\begin{align}\label{eq:alpha^pm}
\alpha^\pm:= \theta(\lambda+1)+t \pm 2\sqrt{\theta\lambda(\theta+t)}.
\end{align}
The Stieltjes-inversion formula (see \cite[Theorem F.6]{Sch12}) implies that
\begin{align}\label{eq:density}
\frac{{\rm d}\mu_{t,\theta,\lambda}}{{\rm d}x}(x) = \frac{t\sqrt{(x-\alpha^-)(\alpha^+-x)}}{2\pi \theta x (x +t(\lambda-1))} \mathbf{1}_{[\alpha^-,\alpha^+]}(x).
\end{align}
Since 
\begin{align*}
\lim_{z\to 0} z G_{\mu_{t,\theta,\lambda}}(z) 
= \frac{-t(t-\theta(\lambda-1)) + t|t-\theta(\lambda-1)|}{2t\theta(\lambda-1)},
\end{align*}
we have
\begin{align}\label{eq:atom}
\mu_{t,\theta,\lambda}(\{0\}) = \begin{cases}
0, & 1 \le \lambda \le 1+t/\theta\\
\vspace{-3mm}\\
1-\dfrac{t}{\theta(\lambda-1)}, & \lambda > 1+t/\theta.
\end{cases}
\end{align}
In particular, $\mu_{t,\theta,\lambda}$ has no singular continuous part since it is freely infinitely divisible (see \cite[Theorem 3.4]{BB04}). We summarize the above result as follows.
\begin{prop}[Density and atom]
For $t,\theta>0$ and $\lambda \ge1$, we get
$$
\mu_{t,\theta,\lambda}(\rmd x) = \max\left\{0, 1-\dfrac{t}{\theta(\lambda-1)}\right\} \delta_0 (\rmd x) + \frac{t\sqrt{(x-\alpha^-)(\alpha^+-x)}}{2\pi \theta x (x +t(\lambda-1))} \mathbf{1}_{[\alpha^-,\alpha^+]}(x)\rmd x.
$$
\end{prop}

Next, we compute the moments of $\mu_{t,\theta,\lambda}$:
$$
m_n(\mu_{t,\theta,\lambda})= \int_\R x^n \mu_{t,\theta,\lambda}({\rm d}x), \qquad n\ge 1.
$$
To obtain $m_n(\mu_{t,\theta,\lambda})$, we first compute its $n$-th {\it free cumulant} $\kappa_n(\mu_{t,\theta,\lambda})$, which is defined as the coefficient of $z^n$ in the power series expansion of the R-transform $R_{\mu_{t,\theta,\lambda}}(z)$. By the computation in the proof of Proposition \ref{prop:relation_freeMeixner}, we have
\begin{align*}
R_{\mu_{t,\theta,\lambda}}(z) 
&=tz \int_{\R} \frac{1}{1-zx} k_{\theta,\lambda}(x)\rmd x =  \sum_{n=0}^\infty tm_n(\pi_{\theta,\lambda})z^{n+1},
\end{align*}
where $m_0(\pi_{\theta,\lambda})=1$. Comparing the coefficients of $z^n$ then yields the following result:
\begin{align}
\kappa_1(\mu_{t,\theta,\lambda}) &= t; \label{eq:free_cumulant}\\
\kappa_{n+1}(\mu_{t,\theta,\lambda}) &= t m_n(\pi_{\theta,\lambda}) = \frac{t\theta^{n}}{n} \sum_{k=0}^{n-1} \binom{n}{k}\binom{n}{k+1} \lambda^k, \qquad n\ge1. \label{eq:free_cumulant2}
\end{align}

\begin{prop}[Moments]\label{prop:moments_GFMG}
Consider $t,\theta>0$ and $\lambda\ge 1$. Then $m_1(\mu_{t,\theta,\lambda})=t$ and for $n\ge 2$,
\begin{align*}
m_n(\mu_{t,\theta,\lambda})=\sum_{m=1}^n \sum_{\substack{r_1,\dots, r_n\ge 0\\ r_1+\cdots+ r_n=m \\ r_1+ 2 r_2 + \cdots + nr_n=n}}P_m^{(n)}(r_1,\dots, r_n)t^m \theta^{n-m} \prod_{s=1}^{n-1} \left( \frac{1}{s}\sum_{k=0}^{s-1} \binom{s}{k}\binom{s}{k+1}\lambda^{k}\right)^{r_{s+1}},
\end{align*}
where
$$
P_m^{(n)}(r_1,\dots, r_n):=\frac{n!}{r_1!r_2!\cdots r_n! ( n - m + 1)!}
$$
\end{prop}
\begin{proof}
One can observe that $m_1(\mu_{t,\theta,\lambda})=\kappa_1(\mu_{t,\theta,\lambda})=t$ by \eqref{eq:free_cumulant}. Let us consider $n\ge 2$. It is known that the number of non-crossing partitions with $r_1$ blocks of size $1$, $r_2$ blocks of size $2$, $\dots$, $r_n$ blocks of size $n$ equals $P_m^{(n)}(r_1,\dots, r_n)$, where $r_1+r_2+\cdots + r_n=m$. By the moment-cumulant formula (cf. \cite[Proposition 11.4]{NS06}) and \eqref{eq:free_cumulant2}, for $n\ge 2$, we obtain
\begin{align*}
m_n(\mu_{t,\theta,\lambda}) 
&= \sum_{\pi \in \mathcal{NC}(n)} \prod_{V\in \pi} \kappa_{|V|}(\mu_{t,\theta,\lambda})\\
&=\sum_{m=1}^n\sum_{\substack{r_1,\dots, r_n\ge 0\\ r_1+\cdots+ r_n=m \\ r_1+ 2 r_2 + \cdots + nr_n=n}}P_m^{(n)}(r_1,\dots, r_n)\kappa_1(\mu_{t,\theta,\lambda})^{r_1}\kappa_2(\mu_{t,\theta,\lambda})^{r_2}\cdots \kappa_n(\mu_{t,\theta,\lambda})^{r_n}\\
&=\sum_{m=1}^n \sum_{\substack{r_1,\dots, r_n\ge 0\\ r_1+\cdots+ r_n=m \\ r_1+ 2 r_2 + \cdots + nr_n=n}}P_m^{(n)}(r_1,\dots, r_n) t^m \theta^{n-m} \prod_{s=1}^{n-1} \left( \frac{1}{s}\sum_{k=0}^{s-1} \binom{s}{k}\binom{s}{k+1}\lambda^{k}\right)^{r_{s+1}}.
\end{align*}
\end{proof}

\begin{ex}\label{ex:moment_GFMG}
Consider $t, \theta>0$ and $\lambda\ge1$. Let us set $m_n:=m_n(\mu_{t,\theta,\lambda})$ for short. By Theorem \ref{prop:moments_GFMG}, we raise the first four moments.
\begin{itemize}
\setlength{\itemsep}{-\parsep}
\item $m_1=t$, 
\item $m_2=t^2+\theta t$,
\item $m_3 =t^3+3\theta t^2 + \theta^2(1+\lambda) t$,
\item $m_4=t^4 + 6\theta t^3 + 2\theta^2(3+2\lambda) t^2 + \theta^3(1+3\lambda+\lambda^2) t$.
\end{itemize}
\end{ex}
In Section \ref{sec4.2}, we will notice that the measure $\mu_{t,\theta,\lambda}$ coincides with a certain scaled free beta prime distribution introduced by \cite{Yos20} for $t,\theta>0$ and $\lambda>1$. According to \cite[Theorem 6.1]{Yos20}, another combinatorial representation of $m_n(\mu_{t,\theta,\lambda})$ will be obtained. See Corollary \ref{cor:moment_another} later.


\subsection{Free selfdecomposability and unimodality}
\label{sec3.3}
One can easily see free selfdecomposability for the measure $\mu_{t,\theta,\lambda}$.

\begin{prop}[Free selfdecomposability]\label{prop_FSD}
Let us consider $t,\theta>0$ and $\lambda\ge1$. The measure $\mu_{t,\theta,\lambda}$ is freely selfdecomposable if and only if $\lambda=1$.
\end{prop}
\begin{proof}
If $\lambda =1$, then the function $tk_{\theta,\lambda} (x)$ is non-increasing on $(0,\infty)$, and therefore the measure $\mu_{t,\theta,1}$ is freely selfdecomposable for all $t,\theta>0$. For $\lambda>1$, the function $k_{\theta,\lambda}(x)$ is supported on $(a^-,a^+)$ and $a^->0$. Hence $\mu_{t,\theta,\lambda}$ is not freely selfdecomposable for any $t>0$ and $\theta>0$.
\end{proof}

A probability measure $\mu$ on $\R$ is said to be {\it unimodal} if there exist $a\in \R$ and a density function $f$ which is nondecreasing on $(-\infty, a)$ and nonincreasing on $(a,\infty)$, such that
$$
\mu ({\rm d} x) = \mu (\{a\}) \delta_a + f(x) {\rm d}x.
$$

According to \cite[Theorem 1]{HT16}, every freely selfdecomposable distribution is unimodal. Hence, $\mu_{t,\theta,1}$ is unimodal for any $t,\theta>0$ by Proposition \ref{prop_FSD}. For given $t,\theta>0$, we investigate the values of $\lambda$ for which $\mu_{t,\theta,\lambda}$ remains unimodal.

\begin{prop}[Unimodality]\label{prop:unimodality}
For given $t,\theta>0$, the measure $\mu_{t,\theta,\lambda}$ is unimodal if and only if $1\le \lambda \le1+t/\theta$.
\end{prop}
\begin{proof}
If $1\le \lambda < 1+t/\theta$, we get $\mu_{t,\theta,\lambda}({\rm d}x) = \frac{t}{2\pi \theta} f(x)dx$ by \eqref{eq:density}, where
$$
f(x)=\frac{\sqrt{(x-\alpha^-)(\alpha^+-x)}}{x(x+t(\lambda-1))}, \qquad x\in (\alpha^-,\alpha^+),
$$
and $\alpha^\pm$ is defined by \eqref{eq:alpha^pm}.
By elementary calculus, we obtain
$$
f'(x)=\frac{k(x)}{2x^2(x+t(\lambda-1))^2\sqrt{(x-\alpha^-)(\alpha^+-x)}},
$$
where
$$
k(x):=2x^3 -3(\alpha^++\alpha^-)x^2+\{4\alpha^+\alpha^- -t(\lambda-1)(\alpha^++\alpha^-)\}x+2\alpha^+\alpha^-t(\lambda-1).
$$
We show that there exists a unique solution $x \in (\alpha^-,\alpha^+)$ of the equation $k(x)=0$. Since
\begin{align*}
k(\alpha^-)&=\{\alpha^- +t (\lambda-1) \}\alpha^-(\alpha^+-\alpha^-)>0,\\
k(\alpha^+)&=\{\alpha^+ +t (\lambda-1) \}\alpha^+(\alpha^--\alpha^+)<0,
\end{align*}
it follows from the intermediate value theorem that there exists at least one solution $x\in (\alpha^-,\alpha^+)$ to the equation $k(x)=0$. Next, we establish the uniqueness of solutions to $k(x)=0$. To this end, we show that the function $k(x)$ is monotone on the interval $(\alpha^-, \alpha^+)$. Since
\begin{align*}
k'(x) 
&= 6x^2-6(\alpha^++\alpha^-) x + \{4\alpha^+\alpha^- - t(\lambda-1)(\alpha^++\alpha^-)\}\\
&= 6\left(x-\frac{\alpha^++\alpha^-}{2}\right)^2 -\frac{3}{2}(\alpha^++\alpha^-)^2 + 4\alpha^+\alpha^- -t(\lambda-1)(\alpha^++\alpha^-)\\
&\le 6\left(\alpha^+-\frac{\alpha^++\alpha^-}{2}\right)^2 -\frac{3}{2}(\alpha^++\alpha^-)^2 + 4\alpha^+\alpha^- -t(\lambda-1)(\alpha^++\alpha^-)\\
&=-2\alpha^+\alpha^- - t(\lambda-1)(\alpha^++\alpha^-)<0,
\end{align*}
the function $k(x)$ is strictly decreasing on $(\alpha^-,\alpha^+)$. Hence the equation $k(x)=0$ has a unique solution in $(\alpha^-,\alpha^+)$, denoted by $x_0\in (\alpha^-,\alpha^+)$. Consequently, the function $f(x)$ is strictly increasing on $(\alpha^-,x_0)$ and strictly decreasing on $(x_0,\alpha^+)$, implying that $\mu_{t,\theta,\lambda}$ is unimodal with mode $x_0$.

If $\lambda= 1+t/\theta$, then $\alpha^-=0$ and $\alpha^+=4(\theta+t)$. In this case, one can verify that $f'(x)<0$ for all $x\in (0,\alpha^+)$. Hence, $\mu_{t,\theta,1+t/\theta}$ is unimodal with mode $0$.

If $\lambda> 1+t/\theta$, then $\mu_{t,\theta,\lambda}$ has an atom at $0$ by \eqref{eq:atom}. However, the absolutely continuous part of $\mu_{t,\theta,\lambda}$ possesses a mode $x_0$, as shown above. Therefore, in this case, $\mu_{t,\theta,\lambda}$ is not unimodal.
\end{proof}

\begin{rem}
According to the proof of Proposition \ref{prop:unimodality}, if $1\le \lambda <1+t/\theta$, then the density function of $\mu_{t,\theta,\lambda}$ is bounded by $f(x_0)$. In contrast, the density function of $\mu_{t,\theta,1+t/\theta}$ is unbounded since $f(x)\to \infty$ as $x\to 0^+$.
\end{rem}

\subsection{Background driving free L\'{e}vy process}
\label{sec:3.4}
Let $\mu$ be a freely selfdecomposable distribution on $\R$. By \cite[Theorem 6.5]{BNT06}, there exists a free L\'{e}vy process\footnote{See \cite{BNT02} for definition of free L\'{e}vy processes.} $\{Z_t\}_{t\ge0}$ affiliated with some $W^\ast$-probability space such that
$$
\mu = \mathcal{L}\left( \int_0^\infty e^{-t} {\rm d}Z_t \right)
$$
and the free L\'{e}vy measure $\nu$ of the law $\mathcal{L}(Z_1)$ satisfies
$$
\int_{\R\setminus[-1,1]} \log(1+|x|) \nu({\rm d}x) <\infty,
$$
where $\int_0^\infty e^{-t} {\rm d}Z_t$ is the free stochastic integral with respect to $\{Z_t\}_{t\ge 0}$, see \cite[Section 6]{BNT06} and \cite{MS23} for details.
The free L\'{e}vy process $\{Z_t\}_{t\ge0}$ is called the {\it background driving free L\'{e}vy process} of $\mu$.

By Proposition \ref{prop_FSD}, the measure $\mu_{1,\theta,1}$ is freely selfdecomposable. From the construction above, we can then consider the background driving free L\'{e}vy process $\{Z_t\}_{t\ge 0}$ of $\mu_{1,\theta,1}$. 

\begin{lem}\label{lem:z_t}
Let $\{Z_t\}_{t\ge0}$ be the free L\'{e}vy process above. Then
\begin{enumerate}
\setlength{\itemsep}{-\parsep}
\item[\rm (1)] For $t\ge 0$, we have
$$
R_{Z_t}(z) = \frac{tz}{\sqrt{1-4\theta z}}, \qquad z\in \C^-.
$$
\item[\rm (2)] The free L\'{e}vy measure of the law of $Z_t$ is given by
$$
\frac{t}{\pi x\sqrt{x(4\theta -x)}}\mathbf{1}_{(0,4\theta)}(x){\rm d}x.
$$
\end{enumerate}
\end{lem}

\begin{proof}
(1) Recall that
$$
R_{\mu_{1,\theta,1}}(z) =  \frac{1-\sqrt{1-4\theta z}}{2\theta}, \qquad z\in \C^-.
$$
By \cite[Theorem 6.7]{MS23}, we have
\begin{align*}
R_{Z_t}(z) = tz \frac{d}{dz} R_{\mu_{1,\theta,1}}(z) = \frac{tz}{\sqrt{1-4\theta z}}.
\end{align*}
(2) We can further compute the R-transform of $Z_t$ as follows:
\begin{align*}
R_{Z_t}(z) &= \frac{tz}{\sqrt{1-4\theta z}}=-\frac{t}{2\sqrt{\theta}} \frac{-z}{\sqrt{-z+\frac{1}{4\theta}}}\\
&=-\frac{t}{2\sqrt{\theta}} \int_{\frac{1}{4\theta}}^\infty \frac{-z}{-z+x} \cdot \frac{1}{\pi \sqrt{x-\frac{1}{4\theta}}}{\rm d}x = \int_0^{4\theta} \left( \frac{1}{1-zx}-1\right) \frac{t}{\pi x \sqrt{x(4\theta -x)}}{\rm d}x,
\end{align*}
where the third equality follows from \cite[Page 304]{SSV12} or \cite[Example 7.2]{MS23}.
Consequently, the L\'{e}vy measure of $Z_t$ is $\frac{t}{\pi x \sqrt{x(4\theta -x)}} \mathbf{1}_{(0,4\theta)}(x){\rm d}x$.
\end{proof}

\begin{rem}
In \cite[Example 7.2]{MS23}, the R-transform and the free Lévy measure of the Meixner-type free gamma distribution $\eta(t,\theta)=\mu_{t\theta,\theta,1}$ were already investigated. In fact, the above lemma can be regarded as a generalization of \cite[Example 7.2]{MS23}.
\end{rem}

The regularity properties of the law of $Z_t$ can be analyzed as follows.
\begin{cor}
For any $t>0$, the law $\mathcal{L}(Z_t)$ is absolutely continuous with respect to Lebesgue measure with continuous density on $\R$.
\end{cor}
\begin{proof}
Let $\nu_t$ be the free L\'{e}vy measure of $\mathcal{L}(Z_t)$. Due to Lemma \ref{lem:z_t} (2), we have
\begin{align*}
\nu_t(\R)
&=\frac{t}{\pi} \int_0^{4\theta} x^{-\frac{3}{2}} (4\theta- x)^{-\frac{1}{2}} {\rm d}x = \infty.
\end{align*}
According to \cite[Theorem 3.4]{HS17}, the measure $\mathcal{L}(Z_t)$ is absolutely continuous with respect to Lebesgue measure with continuous density on $\R$.
\end{proof}

Further, we can obtain the $n$-th free cumulant of the law of $Z_t$.
\begin{prop}
For $n\ge1$, we have
$$
\kappa_1(Z_t)= t \qquad \text{and} \qquad \kappa_n(Z_t) = t(2\theta)^{n-1} \frac{(2n-3)!!}{(n-1)!} \quad \text{for} \quad  n\ge 2.
$$
\end{prop}
\begin{proof}
By Lemma \ref{lem:z_t} (2), for $z$ small enough (more strictly, $|z|<1/4\theta$), we have
\begin{align*}
R_{Z_t}(z) 
&= z \int_0^{4\theta} \frac{1}{1-zx} \frac{t}{\pi\sqrt{x(4\theta-x)}}{\rm d}x\\
&=z \int_0^1 \frac{1}{1-4\theta z u} \cdot \frac{t}{\pi\sqrt{u(1-u)}}{\rm d}u \qquad \text{($x=4\theta u$)}\\
&=\frac{tz}{\pi} \int_0^1 u^{-\frac{1}{2}} (1-u)^{-\frac{1}{2}}(1-4\theta zu)^{-1}{\rm d}u\\
&=\frac{tz}{\pi} \frac{\Gamma(\frac{1}{2})^2}{\Gamma(1)} {}_2 F_1\left(1,\frac{1}{2},1; 4\theta z \right) \qquad \text{(by Euler integral representation)}\\
&=tz \sum_{n=0}^\infty \frac{(1)^{(n)}(\frac{1}{2})^{(n)}}{(1)^{(n)} n!} (4\theta z)^n \qquad \text{($(x)^{(n)}:=x(x+1)\cdots (x+n-1)$)}\\
&=tz + \sum_{n=2}^\infty t (2\theta)^{n-1}\frac{(2n-3)!!}{(n-1)!} z^n.
\end{align*}
Comparing the coefficients of $z^n$ then yields the desired result.
\end{proof}

\subsection{Correlation of a free gamma process}
\label{sec:3.5}

In noncommutative setting, we can consider covariance and correlation as follows. Let $(\mathcal{A},\varphi)$ be a $C^\ast$-probability space and $x,y\in \mathcal{A}$. Then their covariance is defined by
$$
\text{Cov}(x,y) := \varphi(xy)-\varphi(x)\varphi(y).
$$
It is easy to see that $\text{Cov}(x,y)=0$ if $x,y$ are free. Next, their correlation can be define by
$$
\text{Corr}(x,y) := \frac{\text{Cov}(x,y)}{\sqrt{\kappa_2(x)}\sqrt{\kappa_2(y)}},
$$
when $x,y$ have non-zero $2$nd free cumulant (variance). In general, we note that $\text{Corr}(x,y) \neq \text{Corr}(y,x)$ since $xy\neq yx$.

Let $\{X_t\}_{t\ge0}$ be a stochastic process in a $C^\ast$-probability space. The process $\{X_t\}_{t\ge0}$ is called a {\it Meixner-type free gamma process} if it is a free L\'{e}vy process whose marginal distribution at time $1$ is $\mu_{1,\theta,\lambda}$ for some $\theta>0$ and $\lambda\ge1$. By the definition of free L\'{e}vy processes, it follows that $X_t\sim \mu_{t,\theta,\lambda}$ for $t>0$. Below, we compute the correlation of a free gamma process.

\begin{prop}[Correlation]\label{prop:corr}
Let $\{X_t\}_{t\ge0}$ be Meixner-type free gamma process in a $C^\ast$-probability space $(\mathcal{A},\varphi)$. For any $s,t>0$, we have
\begin{align*}
\text{Corr}(X_s, X_t)= \text{Corr}(X_t,X_s)= \sqrt{\frac{s}{t}}.
\end{align*}
\end{prop}
\begin{proof}
For $s<t$, we have
\begin{align*}
\text{Cov}(X_s,X_t) 
&= \varphi(X_s X_t) - \varphi(X_s)\varphi(X_t)\\
&= \varphi(X_s (X_t-X_s) + X_s^2) -\varphi(X_s)\varphi(X_t)\\
&= \varphi(X_s (X_t-X_s)) + \varphi(X_s^2)- \varphi(X_s)\varphi(X_t).
\end{align*}
By Theorem \ref{prop:moments_GFMG} (or Example \ref{ex:moment_GFMG}), we have $\varphi(X_s)=s$ and $\varphi(X_s^2)=s^2+\theta s$. Since $\{X_t\}_{t\ge0}$ is a free L\'{e}vy process, $X_s$ and $X_t-X_s$ are free, and hence $\varphi(X_s(X_t-X_s))=\varphi(X_s) \varphi(X_t-X_s)$. Moreover, by the definition of free L\'{e}vy process, $X_t-X_s \overset{\rmd}{=} X_{t-s}$. Finally, we get
\begin{align*}
\text{Cov}(X_s,X_t) = s(t-s) + s^2+\theta s -st =\theta s.
\end{align*}
By \eqref{eq:free_cumulant} (or Example \ref{ex:moment_GFMG} again), we obtain $\kappa_2(X_s)=\kappa_2(\mu_{s,\theta,\lambda})=\theta s$. Hence,
$$
\text{Corr}(X_s,X_t) = \frac{\theta s} {\sqrt{\theta s} \sqrt{\theta t}} = \sqrt{\frac{s}{t}}.
$$
Recalling the computation of their covariance, we observe that $\text{Cov}(X_t,X_s)=\text{Cov}(X_s,X_t)$ even if $s<t$. Consequently, it also follows that $\text{Corr}(X_t,X_s)=\text{Corr}(X_s,X_t)$.

If $s=t$, then 
$$
\text{Cov}(X_s,X_s)=\varphi(X_s^2)-\varphi(X_s)^2=s^2+\theta s - s^2 = \theta s,
$$
and therefore $\text{Corr}(X_s,X_s)=1$.
\end{proof}

In classical probability, it is known that the correlation of gamma process $\{G_t\}_{t\ge0}$ is
$$
\text{Corr}(G_s,G_t) = \sqrt{\frac{s}{t}} \qquad \text{for} \qquad s,t>0.
$$
For the reason, Proposition \ref{prop:corr} is entirely analogous to the above classical result.


\section{Free convolution formula and free beta prime distributions}
\label{sec4}

\subsection{S-transform} 
In this section, we compute the S-transform of $\mu_{t,\theta,\lambda}$. 

\begin{lem}\label{prop_Strans_GFMG}
For $t,\theta>0$ and $\lambda \ge1$, we have
$$
S_{\mu_{t,\theta,\lambda}} (z) = \frac{t-\theta z}{t(t+\theta(\lambda-1)z)}, \qquad z \in \left( -1 + \mu_{t,\theta,\lambda}(\{0\}),0 \right).
$$
\end{lem}

\begin{proof}
A straightforward computation together with \eqref{eq:formula_R_S} and \eqref{eq:R-trans_theta} shows that the compositional inverse of $R_{\mu_{t,\theta,\lambda}} $ is
$$
R_{\mu_{t,\theta,\lambda}}^{\langle -1 \rangle} (z) = \frac{z(\theta z-t)}{\theta t (1-\lambda)z -t^2}, \qquad z\in (-1+ \mu_{t,\theta,\lambda}(\{0\}),0),
$$
which in turn implies that
\begin{align*}
S_{\mu_{t,\theta,\lambda}} (z) =  \frac{R_{\mu_{t,\theta,\lambda}}^{\langle -1 \rangle} (z)}{z}= \frac{t-\theta z}{t(t+\theta (\lambda-1)z)}.
\end{align*}
\end{proof}

In particular, we show that the measure $\mu_{t,\theta,1}$ is the inverse of Marchenko-Pastur distribution.
\begin{lem}\label{lem:muinverse_formula}
For $t,\theta>0$, we get $\mu_{t,\theta,1} = \pi _{\frac{\theta}{t^2}, 1+\frac{t}{\theta}}^{\langle-1\rangle}$.
\end{lem}
\begin{proof}
The desired formula follows from  the S-transform of $\pi_{\theta/t^2, 1+t/\theta}^{\langle -1\rangle}$. Actually, from Example \ref{ex:inverse_MP} and Lemma \ref{prop_Strans_GFMG}, we have
\begin{align*}
S_{\pi_{\theta/t^2, 1+t/\theta}^{\langle -1\rangle}}(z) = \frac{\theta}{t^2}\left(\frac{t}{\theta}-z\right) = S_{\mu_{t,\theta,1}}(z).
\end{align*}
\end{proof}

\subsection{Free convolution formula}
\label{sec4.2}

Using the S-transform, we can analyze the effect of the third parameter $\lambda$ on the measure $\mu_{t,\theta,\lambda}$ as follows.

\begin{thm}[Free convolution formula for $\mu_{t,\theta,\lambda}$]\label{thm:formula_free_multiplicative}
Consider $t,\theta>0$. If $\lambda>1$, then
\begin{align}
\mu_{t,\theta,\lambda} 
&=D_{t(\lambda-1)}\left(\pi_{1,\frac{t}{\theta(\lambda-1)}} \boxtimes (\pi_{1,1+\frac{t}{\theta}})^{\langle -1\rangle} \right) \label{eq:mu_freebetaprime}\\
&= \mu_{t,\theta,1} \boxtimes \pi_{q^{-1}, q}, \label{eq:mu_anotherrep}
\end{align}
where $q := \frac{t}{\theta(\lambda-1)}$. In particular, $\mu_{t,\theta,1+t/\theta}= \mu_{t,\theta,1} \boxtimes \pi_{1,1}$. Hence the measure $\mu_{t,\theta,1+t/\theta}$ belongs to the class of free compound Poisson distributions.
\end{thm}
\begin{proof}
By Examples \ref{ex_MPlaw}, \ref{ex:inverse_MP} and Lemma \ref{prop_Strans_GFMG}, we have
\begin{align*}
S_{ D_{t(\lambda-1)}\left(\pi_{1,\frac{t}{\theta(\lambda-1)}} \boxtimes (\pi_{1,1+\frac{t}{\theta}})^{\langle -1\rangle} \right)}(z) 
&=\frac{1}{t(\lambda-1)} \frac{1}{\frac{t}{\theta(\lambda-1)}+z} \left(\frac{t}{\theta}-z\right)\\
&=\frac{t-\theta z}{t(t+\theta(\lambda-1)z)}\\
&=S_{\mu_{t,\theta,\lambda}}(z).
\end{align*}
Thus the equation \eqref{eq:mu_freebetaprime} holds. Since $\pi_{\theta,\lambda}=D_\theta(\pi_{1,\lambda})$, we obtain
\begin{align*}
\mu_{t,\theta,\lambda}
&=D_{t(\lambda-1)}\left(\pi_{1,\frac{t}{\theta(\lambda-1)}} \boxtimes (\pi_{1,1+\frac{t}{\theta}})^{\langle -1\rangle} \right) \\
&=D_{t(\lambda-1)} \circ D_{\frac{t}{\theta(\lambda-1)}} \left( \pi_{\frac{\theta(\lambda-1)}{t},\frac{t}{\theta(\lambda-1)}} \boxtimes (\pi_{1,1+\frac{t}{\theta}})^{\langle -1\rangle} \right)\\
&=\pi_{q^{-1},q}\boxtimes (\pi_{\frac{\theta}{t^2},1+\frac{t}{\theta}})^{\langle-1\rangle}\\
&=\pi_{q^{-1},q}\boxtimes  \mu_{t,\theta,1},
\end{align*}
where the last equality follows from Lemma \ref{lem:muinverse_formula}.
\end{proof}

According to Theorem \ref{thm:formula_free_multiplicative}, for any $t,\theta>0$ and $\lambda>1$, the measure $\mu_{t,\theta,\lambda}$ coincides with 
\begin{align}\label{eq:mu_freebetaprime_relation}
D_{t(\lambda-1)}\left(f\beta' \left(\frac{t}{\theta(\lambda-1)}, 1+\frac{t}{\theta} \right)\right),
\end{align}
where 
$$
f\beta'(a, b): = \pi_{1,a}\boxtimes \pi_{1,b}^{\langle -1 \rangle}, \qquad a>0 \text{ and } b>1,
$$ 
is the {\it free beta prime distribution}, introduced by Yoshida \cite[Section 3.4]{Yos20}. Thus, by using \cite[Theorem 6.1]{Yos20}, we obtain a combinatorial formula for the moments of $\mu_{t,\theta,\lambda}$. 

\begin{cor}\label{cor:moment_another}
For $t,\theta>0$ and $\lambda>1$, the $n$-th moment of $\mu_{t,\theta,\lambda}$ is given by
$$
m_n(\mu_{t,\theta,\lambda}) = \theta^n \sum_{\pi \in \mathcal{NCL}(n)}\lambda^{|\pi|-\text{sg}(\pi)} \left(\frac{t}{\theta}\right)^{|\pi|-\text{dc}(\pi)},
$$
where $\mathcal{NCL}(n)$ is the set of all non-crossing linked partitions of $\{1,\dots,n\}$ (see \cite{Dyk07}), ${\rm sg}(\pi)$ is the number of singletons in $\pi$ and ${\rm dc}(\pi)$ is the number of doubly covered elements by $\pi$ (see \cite[Definition 5.7]{Yos20}).
\end{cor}


\subsection{Free beta prime distributions}
\label{sec6.2}

In the previous section, we observed that $\mu_{t,\theta,\lambda}$ is a suitably scaled free beta prime distribution when $\lambda>1$. Using the fact, we now investigate the free beta prime distribution $f\beta'(a,b)$ for any $a>0$ and $b>1$. A straightforward computation of the S-transform yields the following formula.

\begin{prop}\label{prop:formula:freebetaprime_freegamma}
For any $a>0$ and $b>1$, we have
$$
f\beta'(a,b)= \mu_{\frac{a}{b-1}, \frac{a}{(b-1)^2}, \frac{a+b-1}{a}}.
$$
\end{prop}
\begin{proof}
It is easy to see that
\begin{align}\label{eq:Stras_fbetaprime}
S_{f\beta'(a,b)}(z)= \frac{b-1-z}{a+z} \qquad a>0, \ b>1.
\end{align}
To determine $t,\theta>0$ and $\lambda >1$ from $a>0$ and $b>1$, we compare the following equation
$$
\frac{b-1-z}{a+z} = \frac{t-\theta z}{t(t+\theta(\lambda-1)z)},
$$
where the RHS is the S-transform of some $\mu_{t,\theta,\lambda}$. Equivalently,
\begin{align*}
t(\lambda-1)=1, \quad -\theta a + t = -t^2 +t \theta(\lambda-1)(b-1) \quad \text{and} \quad ta  = t^2(b-1).
\end{align*}
Thus, one can see that $t=\frac{a}{b-1}$, $\theta= \frac{a}{(b-1)^2}$ and $\lambda=\frac{a+b-1}{a}$, as desired.
\end{proof}

From Proposition \ref{prop:formula:freebetaprime_freegamma}, we can identify analytic properties of the free beta prime distribution that were not discussed in \cite{Yos20}.

\begin{cor}\label{cor:freebetaprime_FSD}
Let us consider $a>0$ and $b>1$ as follows. Then
\begin{enumerate}[\rm (1)]
\setlength{\itemsep}{-\parsep}
\item The free L\'{e}vy measure of $f\beta'(a,b)$ is given by
$$
\frac{a}{b-1} \frac{k_{u,v}(x)}{x}\rmd x,
$$
where $u=\frac{a}{(b-1)^2}$ and $v=\frac{a+b-1}{a}$. Recall that $k_{u,v}(x)$ is the density function of the Marchenko-Pastur distribution $\pi_{u,v}$.
\item $f\beta'(a,b)$ is not freely selfdecomposable.
\item $f\beta'(a,b)$ is unimodal if and only if $a\ge 1$.
\end{enumerate}
\end{cor}
\begin{proof}
The free L\'{e}vy measure of $f\beta'(a,b)$ follows directly from the definition of $\mu_{t,\theta,\lambda}$ and Proposition \ref{prop:formula:freebetaprime_freegamma}. Since $\frac{a+b-1}{a}>1$, Proposition \ref{prop_FSD} together with Proposition \ref{prop:formula:freebetaprime_freegamma} implies that $f\beta'(a,b)= \mu_{\frac{a}{b-1}, \frac{a}{(b-1)^2}, \frac{a+b-1}{a}}$ is not freely selfdecomposable. Moreover, by Proposition \ref{prop:unimodality} and Proposition \ref{prop:formula:freebetaprime_freegamma}, the measure $f\beta'(a,b)$ is unimodal if and only if
$$
1 <  \frac{a+b-1}{a} \le 1+ \frac{\frac{a}{b-1}}{\frac{a}{(b-1)^2}} = b, 
$$
which is equivalent to $a\ge 1$.
\end{proof}


\section{Potential correspondence}
\label{sec5}

Let $t,\theta>0$ and $\lambda\ge1$ as follows. We consider the following potential function on $\R_{>0}$:
\begin{align*}
V_{t,\theta,\lambda} (x):= 
\begin{cases}
\left(2+\dfrac{t}{\theta}\right) \log x + \dfrac{t^2}{\theta x}, & \lambda=1,\\
\left(1-\dfrac{t}{\theta(\lambda-1)} \right) \log x + \left( 1 + \dfrac{t\lambda}{\theta(\lambda-1)}\right) \log (x+t(\lambda-1)), & \lambda>1.
\end{cases}
\end{align*}
One can verify that
$$
\calH \mu_{t,\theta,\lambda}(x) = \frac{1}{2}V_{t,\theta,\lambda}'(x), \qquad x \in [\alpha_-, \alpha_+]
$$
where $\alpha_{\pm}$ was defined in \eqref{eq:alpha^pm}. First, we analyze the Gibbs measure associated with $V_{t,\theta,\lambda}$ in order to investigate analogous properties for $\mu_{t,\theta,\lambda}$. The ultimate goal of this section is to determine whether $\mu_{t,\theta,\lambda}$ is the equilibrium measure of the free entropy associated with $V_{t,\theta,\lambda}$.

\subsection{Gibbs measure associated with potential}
\label{sec5.1}

We study the Gibbs measure associated with potential $V_{t,\theta,\lambda}$:
$$
\rho_{t,\theta,\lambda}({\rm d}x) = \frac{1}{\mathcal{Z}_{t,\theta,\lambda}} \exp\{ -V_{t,\theta,\lambda}(x)\} {\rm d}x \quad \text{for} \quad t,\theta>0, \ \lambda\ge1,
$$
where $\mathcal{Z}_{t,\theta,\lambda}$ is the normalized constant (i.e., the partition function). We first present an explicit formula for the partition function $\mathcal{Z}_{t,\theta,\lambda}$ as follows.
\begin{lem}\label{prop:Z}
Let us consider $t,\theta>0$ and $\lambda\ge1$. Then
\begin{align*}
\mathcal{Z}_{t,\theta,\lambda}= 
\begin{cases}
\left(\dfrac{t^2}{\theta}\right)^{-1-\frac{t}{\theta}} \Gamma\left(1+\dfrac{t}{\theta}\right), & \lambda=1\\
(t(\lambda+1))^{-\frac{t}{\theta}-1} B\left(\dfrac{t}{\theta(\lambda-1)}, 1+\dfrac{t}{\theta}\right), & \lambda>1.
\end{cases}
\end{align*}
\end{lem}

\begin{proof}
A simple computation leads to the desired results. Actually, if $\lambda=1$, then
\begin{align*}
\calZ_{t,\theta,\lambda} 
&= \int_0^\infty x^{-(2+\frac{t}{\theta})} e^{-\frac{t^2}{\theta x}}{\rm d}x\\
&= \left(\frac{t^2}{\theta}\right)^{-1-\frac{t}{\theta}} \int_0^\infty u^{\frac{t}{\theta}}e^{-u}{\rm d}u \qquad \text{(by putting $u= t^2/(\theta x)$)}\\
&=\left(\frac{t^2}{\theta}\right)^{-1-\frac{t}{\theta}} \Gamma \left(1+\frac{t}{\theta}\right).
\end{align*}
If $\lambda>1$, then
\begin{align*}
\calZ_{t,\theta,\lambda}
&= \int_0^\infty x^{-1 +\frac{t}{\theta(\lambda-1)}} \left(x +t(\lambda-1)\right)^{-1-\frac{t\lambda}{\theta(\lambda-1)}}{\rm d}x\\
&=\int_0^\infty u^{\frac{t}{\theta}} \left(1+t(\lambda-1)u\right)^{-1- \frac{t\lambda}{\theta(\lambda-1)}} {\rm d}u \qquad \text{($x=1/u$)}\\
&=\left(t(\lambda+1)\right)^{-1-\frac{t}{\theta}} \int_0^\infty v^{\frac{t}{\theta}} (1+v)^{-1- \frac{t\lambda}{\theta(\lambda-1)}} {\rm d}v \qquad \text{($v=t(\lambda-1)u$)}\\
&=\left(t(\lambda+1)\right)^{-1-\frac{t}{\theta}} B\left(1+\frac{t}{\theta}, \frac{t}{\theta(\lambda-1)}\right).
\end{align*}
By symmetry of beta functions, we obtain the desired result.
\end{proof}
By Lemma \ref{prop:Z}, the Gibbs measure $\rho_{t,\theta,\lambda}$ has the following form:
\begin{align}
\rho_{t,\theta,1} (\rmd x) &=\frac{(\frac{t^2}{\theta})^{1+\frac{t}{\theta}}}{\Gamma(1+\frac{t}{\theta})} x^{-(2+ \frac{t}{\theta})} e^{-\frac{t^2}{\theta x}} {\rm d}x, \label{eq:rho_1}\\
\rho_{t,\theta,\lambda}(\rmd x) &=\frac{(t(\lambda-1))^{\frac{t}{\theta}+1}}{B(\frac{t}{\theta(\lambda-1)},1+\frac{t}{\theta})} x^{-1 + \frac{t}{\theta(\lambda-1)}} \left(x+t(\lambda-1)\right)^{-1-\frac{t\lambda}{\theta(\lambda-1)}}{\rm d}x, \quad \lambda>1 \label{eq:rho_general}.
\end{align}

From the above representation \eqref{eq:rho_1}, the measure $\rho_{t,\theta,1}$ is the inverse gamma distribution. More precisely, 
\begin{align}\label{eq:rho_inversegamma_formula}
\rho_{t,\theta,1} = \gamma\left(1+\frac{t}{\theta},\frac{\theta}{t^2}\right)^{\langle-1\rangle},
\end{align}
where $\gamma(a,b)^{\langle -1\rangle}$ is defined by
$$
\gamma(a,b)^{\langle -1\rangle} = \frac{1}{b^a \Gamma(a)} x^{-a-1} e^{-\frac{1}{bx}} \rmd x, \qquad a,b>0.
$$
It is known that the class $\{\gamma(a,b)^{\langle-1\rangle}:a,b>0\}$ includes the positive $1/2$-classical stable law $\sqrt{\frac{c}{2\pi}} x^{-\frac{3}{2}} e^{-\frac{c}{2x}}{\rm d}x$, $c>0$ (it is also called L{\'e}vy distribution), but the measure $\rho_{t,\theta,1}$ cannot be the positive $1/2$-classical stable law for any $t,\theta>0$.

For $\lambda>1$, the formula \eqref{eq:rho_general} implies that
\begin{align}
\rho_{t,\theta,\lambda} &=  D_{t(\lambda-1)}\left( \beta' \left(\frac{t}{\theta(\lambda-1)}, 1+\frac{t}{\theta} \right)\right) \label{eq:rho_betaprime} \\
&=D_{t(\lambda-1)} \left( \gamma\left(\frac{t}{\theta(\lambda-1)},1\right) \circledast \gamma\left(1+\frac{t}{\theta},1\right)^{\langle -1\rangle}\right), \label{eq:rho_decom}
\end{align}
where $\beta'(a,b)$ is the beta prime distribution, that is, 
$$
\beta'(a,b) := \gamma(a,1)\circledast \gamma(b,1)^{\langle-1\rangle} = \frac{1}{B(a,b)} x^{-1+a} (1+x)^{-a-b} \rmd x, \qquad a,b>0.
$$
The equation \eqref{eq:rho_decom} is obtained by replacing $\boxtimes$, $\pi_{1,a}$ and $\mu_{t,\theta,\lambda}$ in the equation \eqref{eq:mu_freebetaprime} of Theorem \ref{thm:formula_free_multiplicative} with $\circledast$, $\gamma(a,1)$ and $\rho_{t,\theta,\lambda}$, respectively. Using $\gamma(a,b) = D_b(\gamma(a,1))$ for $a,b>0$, and equations \eqref{eq:rho_inversegamma_formula} and \eqref{eq:rho_decom}, we get
\begin{align*}
\rho_{t,\theta,\lambda} = \rho_{t,\theta,1} \circledast \gamma\left(\frac{t}{\theta(\lambda-1)}, \frac{\theta(\lambda-1)}{t}\right).
\end{align*}
This formula can be interpreted as the classical counterpart of the equation \eqref{eq:mu_anotherrep}. In particular, if we put $\lambda=1+t/\theta$, then
$$
\rho_{t,\theta,1+t/\theta}=\rho_{t,\theta,1} \circledast \gamma(1,1).
$$
The measure $\rho_{t,\theta, 1+t/\theta}$ belongs to the ME (mixture of exponential distributions), that is, the distribution of the form $EZ$, where $E,Z$ are independent random variables such that $E$ is distributed as some exponential distribution and $Z\ge 0$, see \cite{Gol67,Ste67} for details. According to \cite{B92} and \cite{IK79}, the inverse gamma distributions and the beta prime distributions are classically selfdecomposable. Thus, for any $t,\theta>0$ and $\lambda\ge1$, the measure $\rho_{t,\theta,\lambda}$ is selfdecomposable. 

Summarizing the discussion so far, we arrive at the following theorem.
\begin{thm}[Convolution formula for $\rho_{t,\theta,\lambda}$]\label{thm:rho_multiformula}
Let us consider $t,\theta>0$ and $\lambda > 1$. Then we obtain
\begin{align*}
\rho_{t,\theta, \lambda} &= D_{t(\lambda-1)} \left( \gamma\left(\frac{t}{\theta(\lambda-1)},1\right) \circledast \gamma\left(1+\frac{t}{\theta},1\right)^{\langle -1\rangle}\right)\\
&=\rho_{t,\theta,1} \circledast \gamma\left(\frac{t}{\theta(\lambda-1)}, \frac{\theta(\lambda-1)}{t}\right),
\end{align*}
and $\rho_{t,\theta,\lambda}$ is selfdecomposable. In particular, $\rho_{t,\theta, 1+t/\theta}= \rho_{t,\theta,1} \circledast \gamma(1,1)$, and hence it belongs to the ME.
\end{thm}

\begin{rem}\label{rem:Pearson}
Let $p_{t,\theta, \lambda}$ be the density function of the measure $\rho_{t,\theta,\lambda}$. We obtain the following ordinary differential equations:
\begin{align*}
\frac{p_{t,\theta,1}'(x)}{p_{t,\theta,1}(x)}+ \frac{x-\frac{t^2}{2\theta+t}}{\frac{\theta}{2\theta+t}x^2}=0.
\end{align*}
and
\begin{align*}
\frac{p_{t,\theta,\lambda}'(x)}{p_{t,\theta,\lambda}(x)} + \frac{(x+ \frac{t(\lambda-1)}{2}) - \frac{t^2(\lambda+1)}{2(2\theta+t)}}{\frac{\theta}{2\theta+t} (x+ \frac{t(\lambda-1)}{2})^2 - \frac{t^2\theta(\lambda-1)^2}{4(2\theta+t)}} = 0, \qquad \lambda>1.
\end{align*}
Thus, we see that the family $\{\rho_{t,\theta,\lambda}: t,\theta>0, \lambda\ge1\} \subset \calP(\R_{\ge0})$ coincides with a subfamily of {\it Pearson distributions} (see cf. \cite[Page 381]{P1895}). 
\end{rem}


\subsection{Equilibrium measure of free entropy}
\label{sec5.2}

In this section, we investigate the maximizer of free entropy associated with the potential $V_{t,\theta,\lambda}$.

\begin{thm}[Maximizer of free entropy]\label{thm:free_entropy}
For each $t,\theta>0$ and $1 \le \lambda <1+t/\theta$, the measure $\mu_{t,\theta,\lambda}$ is a unique maximizer of 
$$
\Sigma_{V_{t,\theta,\lambda}(}\mu):= \iint_{\R_{>0}\times \R_{>0}} \log |x-y|\mu(\rmd x) \mu(\rmd y) -\int_{\R_{>0}} V_{t,\theta,\lambda}(x)\mu(\rmd x),
$$ 
among all $\mu \in \calP(\R_{>0})$.
\end{thm}
\begin{proof}
Thanks to the theory of the energy problem in \cite{Joh98}, the existence and uniqueness for the maximizer $\mu_V$ of $\Sigma_{V_{t,\theta,\lambda}}$ are guaranteed. The rest of proof is to conclude that $\mu_V=\mu_{t,\theta,\lambda}$. Since $V_{t,\theta,\lambda}$ is regular enough  (see \cite[Section 4.2]{Fer08} and \cite[Section 2]{Fer06}), $\mu_V$ has the density $\Phi_V=\frac{\rmd\mu_V}{\rmd x}$ and is compactly supported on $\R_{>0}$, denoted by $[a,b]$ for some $0<a<b$. We divide two cases for $\lambda$ as follows.

\underline{\bf Case of $\lambda=1$}: By definition of $V_{t,\theta,1}$, we can apply \cite[Theorem 1]{Fer06} in the case of $\lambda=-1-t/\theta$, $\alpha=0$ and $\beta=t^2/\theta$ to the points $a,b$ and the density $\Phi_V$. Then $0<a<b$ satisfy
$$
2+\frac{t}{\theta}-\frac{t^2}{\theta}\cdot \frac{a+b}{2ab}=0 \quad \text{and} \quad \sqrt{ab}=t.
$$
Thus we have
\begin{align*}
a= 2\theta+t -2\sqrt{\theta(\theta+t)} = \alpha^- \quad \text{and} \quad b=2\theta+t + 2\sqrt{\theta(\theta+t)} = \alpha^+,
\end{align*}
where $\alpha^\pm$ is defined in \eqref{eq:alpha^pm} for $\lambda=1$. Moreover,
\begin{align*}
\Phi_V(x) &=\frac{1}{2\pi} \sqrt{(x-a)(b-x)} \cdot \frac{\beta}{\sqrt{ab}x^2} \mathbf{1}_{[a,b]}(x)\\
&= \frac{t\sqrt{(x-\alpha^-)(\alpha^+-x)} }{2\pi \theta x^2}  \mathbf{1}_{[\alpha^-,\alpha^+]}(x) = \frac{\rmd\mu_{t,\theta,1}}{\rmd x}(x) ,
\end{align*}
as desired.

\underline{\bf Case of $1<\lambda < 1+t/\theta$}: By \cite[Theorem IV. 1.11]{ST97}, the points $a$ and $b$ satisfy the following singular integral equations:
\begin{align}\label{eq:SIeq}
\frac{1}{\pi} \int_a^b \frac{V_{t,\theta,\lambda}'(x)}{\sqrt{(b-x)(x-a)}} \rmd x=0 \quad \text{and} \quad \frac{1}{\pi} \int_a^b \frac{xV_{t,\theta,\lambda}'(x)}{\sqrt{(b-x)(x-a)}} \rmd x=2.
\end{align}
By using the formulas
$$
\int_a^b \frac{1}{x \sqrt{(b-x)(x-a)} }\rmd x = \frac{\pi}{\sqrt{ab}}\quad \text{and} \quad \int_a^b \frac{1}{\sqrt{(b-x)(x-a)} }\rmd x =\pi,
$$
we obtain
\begin{align*}
\frac{1}{\pi} &\int_a^b \frac{V_{t,\theta,\lambda}'(x)}{\sqrt{(b-x)(x-a)}} \rmd x \\
&= \frac{1}{\pi} \int_a^b \frac{1-\frac{t}{\theta(\lambda-1)}}{x \sqrt{(b-x)(x-a)}}\rmd x +  \frac{1}{\pi} \int_a^b \frac{1+\frac{t\lambda}{\theta(\lambda-1)}}{(x+t(\lambda-1) )\sqrt{(b-x)(x-a)}}\rmd x \\
&= \left(1-\frac{t}{\theta(\lambda-1)} \right) \frac{1}{\sqrt{ab}} + \left(1+\frac{t\lambda}{\theta(\lambda-1)}\right) \frac{1}{\sqrt{(a+t(\lambda-1))(b+t(\lambda-1))}}.
\end{align*}
Moreover, we have
\begin{align*}
\frac{1}{\pi} & \int_a^b \frac{xV_{t,\theta,\lambda}'(x)}{\sqrt{(b-x)(x-a)}} \rmd x\\
&= \frac{1}{\pi} \int_a^b \frac{1-\frac{t}{\theta(\lambda-1)}}{\sqrt{(b-x)(x-a)}}\rmd x +  \frac{1}{\pi} \int_a^b \frac{\left(1+\frac{t\lambda}{\theta(\lambda-1)}\right)x}{(x+t(\lambda-1) )\sqrt{(b-x)(x-a)}}\rmd x \\
&= \left(1-\frac{t}{\theta(\lambda-1)}\right) +\left(1+\frac{t\lambda}{\theta(\lambda-1)}\right)\left\{1- \frac{t(\lambda-1)}{\sqrt{(a+t(\lambda-1))(b+t(\lambda-1))}}\right\}\\
&=2+\frac{t}{\theta} -\left(1+\frac{t\lambda}{\theta(\lambda-1)}\right)\frac{t(\lambda-1)}{\sqrt{(a+t(\lambda-1))(b+t(\lambda-1))}}.
\end{align*}
Therefore, the equations \eqref{eq:SIeq} imply that
\begin{align}\label{eq:double}
\begin{cases}
 \left(1+\dfrac{t\lambda}{\theta(\lambda-1)}\right) \dfrac{1}{\sqrt{(a+t(\lambda-1))(b+t(\lambda-1))}} = \left(\dfrac{t}{\theta(\lambda-1)} -1\right) \dfrac{1}{\sqrt{ab}} \\
\left(1+\dfrac{t\lambda}{\theta(\lambda-1)}\right)\dfrac{1}{\sqrt{(a+t(\lambda-1))(b+t(\lambda-1))}} = \dfrac{1}{\theta(\lambda-1)}.
\end{cases}
\end{align}
By solving the above equations, we have
\begin{align}\label{eq:coef_solution}
a+b = 2(\theta(\lambda+1)+t)\quad \text{and} \quad ab = (\theta(\lambda-1)-t)^2.
\end{align}
This implies that $a,b$ are the solution of $(z-\alpha^+)(z-\alpha^-)=0$, and hence $a= \alpha^-$ and $b=\alpha^+$. 

By using the formula
$$
\text{p.v}\left(\frac{1}{\pi} \int_a^b \frac{1}{u\sqrt{(u-a)(b-u)}} \frac{{\rmd u}}{u-x} \right) =-\frac{1}{\sqrt{ab}x},
$$
we get
\begin{align*}
\text{p.v.}&  \left(\frac{1}{\pi}\int_{\alpha^-}^{\alpha^+} \frac{V_{t,\theta,\lambda}'(u)}{\sqrt{(u-\alpha^-)(\alpha^+-u)}} \frac{\rmd u}{u-x}\right)\\
&=\left(1-\frac{t}{\theta(\lambda-1)}\right)\text{p.v.} \left(\frac{1}{\pi}\int_{\alpha^-}^{\alpha^+} \frac{1}{u\sqrt{(u-\alpha^-)(\alpha^+-u)}} \frac{\rmd u}{u-x}\right)\\
&\hspace{6mm}+\left(1+\frac{t\lambda}{\theta(\lambda-1)}\right) \text{p.v.} \left(\frac{1}{\pi}\int_{\alpha^-}^{\alpha^+} \frac{1}{(u+t(\lambda-1))\sqrt{(u-\alpha^-)(\alpha^+-u)}} \frac{\rmd u}{u-x}\right)\\
&=-\left(1-\frac{t}{\theta(\lambda-1)}\right)\frac{1}{\sqrt{\alpha^+\alpha^-} x}\\
&\hspace{6mm} -\left(1+\frac{t\lambda}{\theta(\lambda-1)}\right) \frac{1}{\sqrt{(\alpha^-+t(\lambda-1))(\alpha^++t(\lambda-1))}}\cdot \frac{1}{x+t(\lambda-1)}\\
&=\frac{1}{\theta(\lambda-1)x} -\frac{1}{\theta(\lambda-1)}\cdot\frac{1}{x+t(\lambda-1)} \qquad \text{(by \eqref{eq:double} and \eqref{eq:coef_solution})}\\
&=\frac{t}{\theta x (x+t(\lambda-1))}.
\end{align*}
From \cite[Theorem IV. 3.1]{ST97}, we finally obtain
\begin{align*}
\Phi_V(x)&=\frac{1}{2\pi} \sqrt{(x-\alpha^-)(\alpha^+ -x)} \times \text{p.v.}  \left( \frac{1}{\pi} \int_{\alpha^-}^{\alpha^+} \frac{V_{t,\theta,\lambda}'(u)}{\sqrt{(u-\alpha^-)(\alpha^+-u)}} \frac{\rmd u}{u-x}\right)\\
&= \frac{t\sqrt{(x-\alpha^-)(\alpha^+ -x)} }{2\pi \theta x (x+t(\lambda-1))} = \frac{\rmd \mu_{t,\theta,\lambda}}{\rmd x}(x),
\end{align*}
as desired.
\end{proof}

\begin{rem}
From the proof of the above theorem, we observe that the class $\{\mu_{t,\theta,1}:t,\theta>0\}$ forms a special subclass of the free GIG distributions (see \cite{Fer06, HS19}).
\end{rem}

Due to Theorem \ref{thm:free_entropy}, the potential correspondence maps the measure $\rho_{t,\theta,\lambda}$ to the measure $\mu_{t,\theta,\lambda}$ for all $t,\theta>0$ and $1\le \lambda <1+t/\theta$. In particular, the potential correspondence maps the beta prime distributions $\beta'(a,b)$ to the free ones $f\beta'(a,b)$ for all $a,b>1$, due to Proposition \ref{prop:formula:freebetaprime_freegamma} and \eqref{eq:rho_betaprime}. Consequently, we obtain the following result for free beta prime distributions.
\begin{cor}
\label{cor:freebetaprime_entropy}
Let us consider $a,b>1$. The measure $f\beta'(a,b)$ is the unique maximizer of the free entropy $\Sigma_{V_{a,b}}(\mu)$ among probability measures $\mu$ on $\R_{>0}$, where
$$
V_{a,b}(x) = (1-a) \log x + (a+b) \log(1+x), \qquad x>0.
$$
\end{cor}

\section{Meixner-type free beta-gamma algebra}
\label{sec:Fbgalg}

In classical probability, there are many algebraic relations between gamma and beta random variables, so called {\it beta-gamma algebra}, see e.g. \cite{FS23}. A purpose of this section is to study algebraic relations between free beta and free gamma random variables in the sense of Meixner-type.

According to Section \ref{sec:intro}, if $G_1^{(p)} \sim \eta(p,1)$ and $G_2^{(q)}\sim \eta(q,1)$ are free, then
$$
G_1^{(p)} + G_2^{(q)} \overset{\rmd}{=} G_3^{(p+q)}, \qquad p,q>0
$$
for some $G_3^{(p+q)}\sim \eta(p+q,1)$. We call a positive operator $G\sim \eta(p,1)$ $(p>0)$ a {\it Meixner-type free gamma random variable}. Recall that $\eta(p,1)=\mu_{p,1,1}$ for all $p>0$.

First, we investigate the reversed measure of $\mu_{t,\theta,\lambda}$ as follows. Since $\mu_{t,\theta,\lambda}(\{0\})=0$ for $1\le \lambda \le 1+t/\theta$, we can define the reversed measure $\mu_{t,\theta,\lambda}^{\langle-1\rangle}$ in this case. We have already obtained the measure $\mu_{t,\theta,1}^{\langle-1\rangle}$ by Lemma \ref{lem:muinverse_formula}.

\begin{lem}\label{thm:reverse}
For $t,\theta>0$ and $1< \lambda \le 1+\frac{t}{\theta}$, we have
$$
\mu_{t,\theta,\lambda}^{\langle -1\rangle} =D_{(t(\lambda-1))^{-1}}( \mu_{t', \theta', \lambda'} ),
$$
where
\begin{align*}
t' = \frac{(\theta+t)(\lambda-1)}{t-\theta(\lambda-1)}, \quad \theta' = \frac{\theta(\theta+t)(\lambda-1)^2}{ (t-\theta(\lambda-1))^2} \quad \text{and} \quad \lambda'=\frac{t\lambda}{(\theta+t)(\lambda-1)}\ge1.
\end{align*}
\end{lem}
\begin{proof}
By Theorem \ref{thm:formula_free_multiplicative} and Proposition \ref{prop:formula:freebetaprime_freegamma}, we obtain
\begin{align*}
\mu_{t,\theta,\lambda}^{\langle-1\rangle} 
&= D_{(t(\lambda-1))^{-1}} \left( \pi_{1,\frac{t}{\theta(\lambda-1)}}^{\langle-1\rangle} \boxtimes \pi_{1,1+\frac{t}{\theta}}\right) \\
&=D_{(t(\lambda-1))^{-1}}\left(f\beta' \left( 1+\frac{t}{\theta}, \frac{t}{\theta(\lambda-1)} \right)\right)=D_{(t(\lambda-1))^{-1}} (\mu_{t',\theta',\lambda'}). 
\end{align*}
\end{proof}

The free infinite divisibility for the reversed measure of $\mu_{t,\theta,\lambda}$ follows from Section \ref{sec3}, Lemmas \ref{lem:muinverse_formula} and \ref{thm:reverse}.
\begin{cor}
Let us consider $t,\theta>0$ and $1\le \lambda \le 1+t/\theta$. Then the following properties hold.
\begin{enumerate}[\rm (1)]
\setlength{\itemsep}{-\parsep}
\item $\mu_{t,\theta,\lambda}^{\langle -1\rangle}$ is freely infinitely divisible. 
\item $\mu_{t,\theta, \lambda}^{\langle -1 \rangle}$ is freely selfdecomposable if and only if $\lambda=1+t/\theta$.
\item $\mu_{t,\theta,\lambda}^{\langle -1\rangle}$ is unimodal.
\end{enumerate}
\end{cor}

We recall that, if $\Gamma_p \sim \gamma(p,1)$ and $\Gamma_q \sim \gamma(q,1)$ are classically independent, then $\frac{\Gamma_p}{\Gamma_q}$ is distributed as beta prime distribution $\beta'(p,q)= \frac{x^{p-1} (1+x)^{-p-q}}{B(p,q)} \rmd x$ for $p,q>0$, see \cite[Chapter 27]{BJK95}. By analogy, we investigate {\it free beta-prime random variables in the sense of Meixner-type}, that is, the product of $G_1^{(p)}$ and $(G_2^{(q)})^{-1}$, where $G_1^{(p)}\sim \eta(p,1)$ and $G_2^{(q)}\sim \eta(q,1)$ are free.

\begin{prop}\label{prop:fMbetaprime}
Given $p,q>0$, we assume that $G_1^{(p)}\sim \eta(p,1)$ and $G_2^{(q)}\sim \eta(q,1)$ are free in a $C^\ast$-probability space $(\mathcal{A},\varphi)$. Then
$$
 (G_2^{(q)})^{-\frac{1}{2}} G_1^{(p)} (G_2^{(q)})^{-\frac{1}{2}} \sim \eta(p,1)\boxtimes \eta(q,1)^{\langle-1 \rangle} = D_{\frac{1+q}{q^2}}\left( \mu_{p,1,1+\frac{p}{1+q}}\right).
$$
\end{prop}
\begin{proof}
Since $\sigma(G_2^{(q)}) = [2+q-2\sqrt{q+1}, 2+q+2\sqrt{q+1}]$ and $f(x)=1/x$ is a bounded continuous function on $\sigma(G_2^{(q)})$, two random variables $G_1^{(p)}$ and $(G_2^{(q)})^{-1}=f(G_2^{(q)}) \in \mathcal{A}$ are also free by the Stone-Weierstrass' theorem.
We observe
\begin{align*}
\eta(p,1)\boxtimes \eta(q,1)^{\langle-1 \rangle} 
&= \mu_{p,1,1} \boxtimes \pi_{q^{-2},1+q} \qquad \text{(by Lemma \ref{lem:muinverse_formula})}\\
&=D_{\frac{1+q}{q^2}} \left(\mu_{p,1,1} \boxtimes \pi_{\frac{1}{1+q}, 1+q} \right)\\
&=D_{\frac{1+q}{q^2}} \left(\mu_{p,1,1+\frac{p}{1+q} }\right) \qquad \text{(by Theorem \ref{thm:formula_free_multiplicative})}.
\end{align*}
\end{proof}

Finally, we investigate the sum of freely independent and identically distributed reciprocal Meixner-type free gamma random variable.
\begin{thm}\label{prop:reversegamma}
Given $p>0$ and $n\in \N$, let us consider freely independent random variables $G_1^{(p)},G_2^{(p)}, \dots, G_{2^n}^{(p)} \sim \eta(p,1)$ in some $C^*$-probability space $(\mathcal{A},\varphi)$. Then
$$
\left(\frac{1}{G_1^{(p)}} + \frac{1}{G_2^{(p)}}+\cdots + \frac{1}{G_{2^n}^{(p)}} \right)^{-1} \overset{\rmd}{=} \left(2^n+\frac{2^n-1}{p}\right)^{-2}G^{(2^np+2^n-1)},
$$
for some $G^{(2^np+2^n-1)} \sim \eta(2^np+2^n-1,1)$.
\end{thm}
\begin{proof}
By an argument similar to that in Proposition \ref{prop:fMbetaprime}, $(G_1^{(p)})^{-1}, (G_2^{(p)})^{-1}, \dots, (G_{2^n}^{(p)})^{-1}\in \mathcal{A}$ are also free. Denote by $\tau_p:=\eta(p,1)^{\langle-1\rangle} = \pi_{p^{-2},1+p}$ for $p>0$. Then the distribution of the LHS coincides with $(\tau_p^{\boxplus 2^n})^{\langle-1\rangle}$. Finally we should prove that
\begin{align}\label{eq:tau}
\tau_p^{\boxplus 2^n} = D_{(2^n + \frac{2^n-1}{p})^2} (\tau_{2^np+2^n-1}).
\end{align}
For $n=1$, we get
$$
\tau_p^{\boxplus 2} = \pi_{p^{-2},1+p}^{\boxplus 2}=D_{p^{-2}}( \pi_{1,2+2p}) = D_{(1+2p)^2p^{-2}} (\pi_{(1+2p)^{-2}, 1+(1+2p)})= D_{(2+p^{-1})^2} (\tau_{1+2p}).
$$
We assume that \eqref{eq:tau} holds true in the case when $n=k$. Then
\begin{align*}
\tau_p^{\boxplus 2^{k+1}} 
&= (\tau_p^{\boxplus 2^k})^{\boxplus 2} = D_{(2^k + \frac{2^k-1}{p})^2} (\tau_{2^kp+2^k-1})^{\boxplus 2}\\
&=D_{(2^k + \frac{2^k-1}{p})^2} D_{(2+ \frac{1}{2^k p +2^k-1})^2} (\tau_{1+2(2^kp +2^k-1)})\\
&=D_{(2^{k+1} + \frac{2^{k+1}-1}{p})^2} (\tau_{2^{k+1}p + 2^{k+1}-1}).
\end{align*}
By induction, we obtain the desired formula \eqref{eq:tau}.
\end{proof}

\begin{cor}
Given $m \ge 2$, we consider free copies $\{ G_1^{((2(m-1))^{-1})}, G_2^{((2(m-1))^{-1})}\}$ from $\eta((2(m-1))^{-1},1)$ and free copies $\{G_1^{((m-1)^{-1})}, \dots, G_m^{((m-1)^{-1})}\}$ from $\eta((m-1)^{-1},1)$. Then
$$
\left( \frac{1}{G_1^{((2(m-1))^{-1})}} + \frac{1}{G_2^{((2(m-1))^{-1})}} \right)^{-1} \overset{\rmd}{=} \frac{1}{4m^2} (G_1^{((m-1)^{-1})}+ \cdots + G_m^{((m-1)^{-1})})
$$
\end{cor}
\begin{proof}
By putting $p=\frac{1}{2(m-1)}$ and $n=1$ in Proposition \ref{prop:reversegamma}, we get
\begin{align*}
\left( \frac{1}{G_1^{(p)}} + \frac{1}{G_2^{(p)}} \right)^{-1} 
&\sim D_{\frac{p^2}{(2p+1)^2}} \eta(2p+1,1)\\
&= D_{\frac{1}{4}\cdot \frac{(2p)^2}{(2p+1)^2}} \eta(2p,1)^{\boxplus \frac{2p+1}{2p}}=D_{\frac{1}{4m^2}} \eta\left(\frac{1}{m-1}, 1 \right)^{\boxplus m},
\end{align*}
as desired.
\end{proof}

In classical probability theory, it is known that, if $\Gamma_p \sim \gamma(p,1)$ and $\Gamma_q \sim \gamma(q,1)$ are independent, then the random variable
$$
\left(1+\frac{\Gamma_p}{\Gamma_q} \right)^{-1} = \frac{\Gamma_p^{-1}}{\Gamma_p^{-1}+\Gamma_q^{-1}}
$$
is distributed as a beta distribution $\beta(p,q)$ for $p,q>0$. Below, we study {\it free beta random variables in the sense of Meixner-type}.

\begin{thm}
Given $p>0$, let us set free random variables $G_1^{(p)}, G_2^{(p)} \sim \eta(p,1)$ in $C^\ast$-probability space $(\mathcal{A},\varphi)$. Define
$$
B^{(p)} := \{(G_1^{(p)})^{-1} + (G_2^{(p)})^{-1} \}^{-\frac{1}{2}} (G_1^{(p)})^{-1}  \{(G_1^{(p)})^{-1} + (G_2^{(p)})^{-1} \}^{-\frac{1}{2}} \in \mathcal{A},
$$
and $\mu_p := \mathcal{L}(B^{(p)})$. Then the following assertions hold.
\begin{enumerate}[\rm (1)]
\setlength{\itemsep}{-\parsep}
\item Its S-transform is given by
$$
S_{\mu_p}(z)=\frac{p^4}{(1+p+z)(2p+1-z)}
$$
for $z$ in a neighborhood of $(-1,0)$.
\item Its R-transform is given by
$$
R_{\mu_p}(z)= \frac{p(z-p^3)-\sqrt{(3p+2)^2z^2-2p^5z+p^8}}{2z}, \qquad z\in \left(-\frac{p^3}{2(p+1)}, 0\right).
$$
\item The measure $\mu_p$ is not freely infinitely divisible for any $p>0$. 
\end{enumerate}
\end{thm}

\begin{proof}
The existence of the measure $\mu_p$ follows from Riesz-Markov-Kakutani's theorem.
Note that, $(G_1^{(p)})^{-1},(G_2^{(p)})^{-1} \sim \pi_{p^{-2}, 1+p}$ are free in $(\mathcal{A},\varphi)$. Hence $(G_1^{(p)})^{-1}+(G_2^{(p)})^{-1}$ and $B^{(p)}$ are also free in $(\mathcal{A},\varphi)$ by free Lukacs property (see \cite{Szp15}). Since
$$
(G_1^{(p)})^{-1} = \{(G_1^{(p)})^{-1}+(G_2^{(p)})^{-1}\}^{\frac{1}{2}} B^{(p)}  \{(G_1^{(p)})^{-1}+(G_2^{(p)})^{-1}\}^{\frac{1}{2}},
$$
we have 
$$
\pi_{p^{-2}, 1+p} = \mu_p \boxtimes D_{\frac{p^2}{(2p+1)^2}} \eta(2p+1,1)
$$
by Theorem \ref{prop:reversegamma}. Therefore we get
\begin{align*}
S_{\mu_p}(z) 
&= \frac{S_{\pi_{p^{-2}, 1+p}}(z)}{S_{D_{p^2(2p+1)^{-2}} (\eta(2p+1,1))}(z)}\\
&=\frac{p^2}{1+p+z} \cdot \frac{p^2}{2p+1-z}=\frac{p^4}{(1+p+z)(2p+1-z)}.
\end{align*}

Next, since $z\mapsto z S_{\mu_p}(z)$ is strictly increasing on $(-1,0)$ for any $p>0$, we have $R_{\mu_p}^{\langle -1\rangle}(z)= z S_{\mu_p}(z)$. Hence, we can compute the R-transform of $\mu_p$ by using the relation. 

The complex equation $(3p+2)^2z^2 -2p^5 z+p^8=0$ has distinct roots
$$
p^4 \cdot \frac{p-2 \sqrt{(p+1)(2p+1)} \ i}{(3p+2)^2} \quad \text{and} \quad p^4 \cdot \frac{p+2 \sqrt{(p+1)(2p+1)}\ i}{(3p+2)^2},
$$
where $i=\sqrt{-1}$. This means that $\sqrt{(3p+2)^2z^2 -2p^5 z+p^8}$ has a pole at $z\in \C^-$. Hence $R_{\mu_p}$ does not have an analytic continuation to $\C^-$. The measure $\mu_p$ is not freely infinitely divisible.
\end{proof}


\section{Asymptotic roots of polynomials and the generalized Meixner-type free gamma distributions}
\label{sec7}

In this section, we briefly discuss the connection between orthogonal (Jacobi / Bessel) polynomials and the measure $\mu_{t,\theta,\lambda}$ via finite free probability, as developed in \cite{Mar21,MSS22}. We begin by introducing the notation that will be used throughout this section.
\begin{itemize}
\setlength{\itemsep}{-\parsep}
\item For a polynomial $p$ of degree $d$, we denote by $\widetilde{e}_k^{(d)}(p)$ the normalized $k$-th elementary symmetric polynomial in the $d$ roots $\lambda_1(p),\dots, \lambda_d(p)$ of $p$. Explicitly,
$$
\widetilde{e}_k^{(d)}(p) : = \binom{d}{k}^{-1} \sum_{1\le i_1< \cdots < i_k \le d} \lambda_{i_1}(p)\cdots \lambda_{i_k}(p), \quad k=1,\dots, d.
$$
Then we can represent any polynomials $p$ of degree $d$ as
$$
p(x) = \prod_{i=1}^d (x-\lambda_i(p))=\sum_{k=0}^d (-1)^k \binom{d}{k} \widetilde{e}_k^{(d)}(p) x^{d-k}.
$$
\item ({\it Dilation}) For $c\neq 0$ and a polynomial $p$ of degree $d$, we define
$$
(D_c(p))(x):= c^d p\left(\frac{x}{c}\right).
$$
Then one can see that $\widetilde{e}_k^{(d)}(D_c(p))= c^k\ \widetilde{e}_k^{(d)}(p)$ for $k=1,\dots, d$.
\item ({\it Empirical root distribution}) For a polynomial $p$ of degree $d$, we define the probability measure
$$
\rho[[ p]] :=\frac{1}{d}\sum_{p(x)=0} \delta_x.
$$
The measure is called the {\it empirical root distribution} of $p$. If $p$ is real-rooted (resp. positive real-rooted), then $\rho[[p]] \in \calP(\R)$ (resp. $\rho[[p]] \in \calP(\R_{>0})$).
\end{itemize}

\subsection{Jacobi polynomial}

For $a \in \R\setminus\{i/d: i=1,\dots, d-1\}$ and $b\in \R$, we denote by $J^{(a,b)}_d$ the monic polynomial of degree $d$, in which its coefficient is given by
$$
\widetilde{e}_k^{(d)}\left(J^{(a,b)}_d \right) :=  \frac{(bd)_k}{(ad)_k} \qquad \text{for} \ k=0, 1,\dots,d,
$$
where $(x)_n:=x(x-1)\cdots(x-n+1)$ and $(x)_0:=1$. The polynomials $J^{(a,b)}_d$ are well-known as the {\it Jacobi polynomials}. According to \cite[(80)]{MFMP24b}, the polynomial $J^{(a,b)}_d$ can be represented by a hypergeometric function as follows:
$$
J_d^{(a,b)}(x)=(-1)^d\frac{ (bd)_d}{(ad)_d} {}_2F_1 (-d, ad-d+1, bd-d+1; x).
$$
It is known that the polynomial is orthogonal with respect to the weight function 
$$
w_{d}^{(a,b)}(x):=x^{d(b-1)}(1-x)^{d(a-b-1)}
$$ 
when $-bd+d-1 \notin \Z_{\ge0}$ (see e.g. \cite[Theorem 1]{DJJ13}).
Recently, the class of hypergeometric polynomials (including $J^{(a,b)}_d$) was studied in the framework of finite free probability theory, see e.g. \cite{MFMP24a, MFMP24b, AFPU24}. 

In particular, we define
$$
\widehat{J}^{(a,b)}_d(x):= J^{(a,b)}_d(-x).
$$
By \cite[Section 5.2]{MFMP24b} and \cite[Proposition 4 and Theorem 5]{DJJ13}, the polynomial $\widehat{J}^{(a,b)}_d$ has $d$ distinct roots in which are all nonnegative when $a<0$ and $b>1$. Furthermore, we define the monic real-rooted polynomial of degree $d$ by
$$
p_d^{(t,\theta,\lambda)}:= D_{t(\lambda-1)} (\widehat{J}^{\ (A, B) }_d), \qquad \text{where} \quad  A=-\frac{t}{\theta} \quad \text{and} \quad B=\frac{t}{\theta(\lambda-1)} +\frac{1}{d},
$$
for $t,\theta>0$ and $1<\lambda \le 1+t/\theta$. In this case, it is easy to check $A<0$ and $B>1$, and therefore $p_d^{(t,\theta,\lambda)}$ also has $d$ distinct roots in which are all nonnegative. Moreover, $p_d^{(t,\theta,\lambda)}$ is orthogonal with respect to the weight function
\begin{align}\label{eq:weight}
w_d^{(A,B)}\left(-\frac{x}{t(\lambda-1)}\right) \propto x^{d(B-1)}(x+t(\lambda-1))^{d(A-B-1)}.
\end{align}

According to recent work \cite{AFPU24}, the ratio of consecutive coefficients of a polynomial plays a role in the S-transform in the framework of finite free probability. In what follows, we apply the results of \cite{AFPU24} to the sequence of polynomials $(p_d^{(t,\theta,\lambda)})_{d\in \N}$. A direct computation shows that
$$
\frac{\widetilde{e}_{k-1}^{(d)}(p_d^{(t,\theta,\lambda)})}{\widetilde{e}_{k}^{(d)}(p_d^{(t,\theta,\lambda)})} = \frac{1}{t(\lambda-1)} \frac{\frac{t}{\theta}+\frac{k-1}{d}}{\frac{t}{\theta(\lambda-1)}-\frac{k-2}{d}}.
$$
As $d\to\infty$ with $k/d\to z\in (0, 1)$, the limiting value of the above consecutive coefficient ratio can be computed as follows:
\begin{align*}
\frac{\widetilde{e}_{k-1}^{(d)}(p_d^{(t,\theta,\lambda)})}{\widetilde{e}_{k}^{(d)}(p_d^{(t,\theta,\lambda)})} &\to \frac{1}{t(\lambda-1)} \frac{\frac{t}{\theta}+z}{\frac{t}{\theta(\lambda-1)}-z} \\
&= \frac{1}{t(\lambda-1)} \frac{(1+\frac{t}{\theta})  -1 - (-z)}{\frac{t}{\theta(\lambda-1)} +(-z)} \\
&= S_{D_{t(\lambda-1)} (f\beta'(\frac{t}{\theta(\lambda-1)}, 1+\frac{t}{\theta}))}(-z) \qquad \text{(by \eqref{eq:Strans_dilation} and \eqref{eq:Stras_fbetaprime})}\\
&= S_{\mu_{t,\theta,\lambda}}(-z) \qquad \text{(by \eqref{eq:mu_freebetaprime_relation})}.
\end{align*}
By \cite[Theorem 1.1]{AFPU24}, the limiting behavior of the empirical root distribution of $ p_d^{(t,\theta,\lambda)}$ can be described as follows.
\begin{prop}\label{prop:lambda_finitefree}
For $t,\theta>0$ and $1<\lambda \le 1+t/\theta$, we have
$$
\rho[[p_d^{(t,\theta,\lambda)}]] \xrightarrow{w} \mu_{t,\theta,\lambda}  \qquad \text{as} \qquad d\to\infty.
$$
\end{prop}

\subsection{Bessel polynomial}
For $a\in \R \setminus \{i/d: i=1, \dots, d-1\}$, we define $B^{(a)}_d$ as the monic polynomial of degree $d$, whose coefficients are given by
$$
\widetilde{e}_k^{(d)}\left(B^{(a)}_d \right) = \frac{d^k}{(a d)_k} \qquad \text{for} \ k=0,1,\dots, d.
$$
The polynomials $B^{(a)}_d$ are well-known as the {\it Bessel polynomials}. According to \cite[(80)]{MFMP24b}, the polynomial $B^{(a)}_d$ can be represented by a hypergeometric polynomial as follows:
$$
B^{(a)}_d(x) = \frac{(-1)^d}{(ad)_d} {}_2 F_0 (-d, a d-d+1, \cdot \ ; x).
$$

We define 
$$
\widehat{B}_d^{(a)}(x):= B_d^{(a)}(-x).
$$
According to \cite{MFMP24b}, $\widehat{B}_d^{(a)}$ has $d$ distinct roots in which are all nonnegative whenever $a<0$. Furthermore, we define the monic real-rooted polynomial of degree $d$ by
$$
p_d^{(t,\theta,1)} : = D_{t^2/\theta}\left(\widehat{B}_d^{(-t/\theta)}\right), \qquad t,\theta>0.
$$
Since $-t/\theta<0$, the polynomial $p_d^{(t,\theta,1)}$ also has $d$ distinct roots in which are all nonnegative. Then we obtain
\begin{align*}
\frac{\widetilde{e}_{k-1}^{(d)}(p_d^{(t,\theta,1)})}{\widetilde{e}_{k}^{(d)}(p_d^{(t,\theta,1)})} = \frac{1}{t^2}\left(t +\theta\cdot \frac{k-1}{d}\right) \to \frac{1}{t^2} (t+\theta z) = S_{\mu_{t,\theta,1}}(-z)
\end{align*}
as $k/d\to z \in (0,1)$. According to \cite{AFPU24}, we obtain the following result.
\begin{prop}\label{prop:lambda=1_finitefree}
For $t,\theta>0$, we get
$$
\rho[[p_d^{(t,\theta,1)}]] \xrightarrow{w} \mu_{t,\theta,1} \qquad \text{as} \qquad d\to \infty.
$$
\end{prop}

\subsection{Final remark}

These results (Proposition \ref{prop:lambda_finitefree} and \ref{prop:lambda=1_finitefree}) are not essentially new since Mart\'{i}nez-Finkelshtein et al. \cite{MFMP24a, MFMP24b} and Arizmendi et al. \cite{AFPU24} have already studied the relationship between the asymptotic behavior of the empirical root distributions of Jacobi or Bessel polynomials and free probability. Nevertheless, we would like to emphasize that our analysis provides deeper insights into the behavior of the roots of Jacobi or Bessel polynomials, thanks to the understanding gained from the distributional properties of the Meixner-type free gamma distribution and its connection with free entropy. For instance, the weight function \eqref{eq:weight} of $p_d^{(t,\theta,\lambda)}$ belongs to the Pearson class, which may be connected to the Gibbs measure $\rho_{t,\theta,\lambda}$ associated with the potential $V_{t,\theta,\lambda}$ discussed in Section \ref{sec5}. On the other hand, to the best of our knowledge, the weight function of $p_d^{(t,\theta,1)}$ is not explicitly known. From the perspective of potential theory, the weight function $w$ can be expected to take the form $w(x)\propto x^{-2-\frac{t}{\theta}} e^{-\frac{t^2}{\theta x}}$. \\
\vspace{3mm}



\hspace{-6mm}{\bf Acknowledgement.}
The authors would like to thank Takahiro Hasebe (Hokkaido University) for fruitful discussions in relation to this project. The authors wish to express their sincere gratitude to the anonymous referee for carefully reading the manuscript and providing numerous valuable comments and suggestions, which have substantially improved the paper. In particular, we are grateful for the referee's observations regarding the relation between the measures $\mu_{t,\theta,\lambda}$ and the centered free Meixner distributions (Proposition~\ref{prop:relation_freeMeixner}), its connection with Pearson distributions (Remark~\ref{rem:Pearson}), and the orthogonality of $p_d^{(t,\theta,\lambda)}$, which have given us deeper insights than originally anticipated. We would like to express our heartfelt thanks once again.

N. S. was supported by JSPS  Grant-in-Aid for Scientific Research (C) Grant No. 23K03133. Y.U. was supported by JSPS Grant-in-Aid for Young Scientists Grant No. 22K13925.

\vspace{10mm}

\hspace{-6mm}{\bf Noriyoshi Sakuma}\\
Department of Mathematics, Graduate School of Science, Osaka University, 1-1 Machikaneyama, Toyonaka 560-0043, Osaka, Japan.\\
E-mail: sakuma@math.sci.osaka-u.ac.jp\\
\vspace{3mm}

\hspace{-6mm}{\bf Yuki Ueda}\\
Faculty of Education, Department of Mathematics, Hokkaido University of Education, 9 Hokumon-cho, Asahikawa, Hokkaido 070-8621, Japan.\\
E-mail: ueda.yuki@a.hokkyodai.ac.jp

\end{document}